\documentclass[11pt]{article}
\usepackage{fullpage}
\usepackage{graphicx}
\usepackage{psfrag}
\usepackage{graphicx,psfrag}
\usepackage{epsfig}
\usepackage{color}
\usepackage{color}
\usepackage{url}
\usepackage{setspace}
\usepackage{amsfonts,amssymb,amsthm}
\usepackage{amsfonts, amssymb}
\usepackage{hyperref}
\newcommand{\mathb}[1]{\mbox{\boldmath$#1$}}
\newcommand{\sumk}{\sum_{i=1}^k}

\newcommand{\vari}{\pi_i (1-\pi_i)}

\newcommand{\hatvari}{\hat \pi_i (1-\hat \pi_i)}

\newtheorem{theorem}{Theorem}
\newtheorem{lemma}{Lemma}
\newtheorem{corollary}{Corollary}

\newtheorem{remark}{Remark}

\begin{document}

\title{ Testing homogeneity of proportions from sparse binomial data with a large number of groups  
}


\author{Junyong Park \thanks{Department of Mathematics and Statistics, University of Maryland Baltimore County, 1000 Hilltop Circle, Baltimore, MD 21250,  U.S.A.\texttt{junpark@umbc.edu}}
}
\date{}


\maketitle

\begin{abstract}
In this paper, we consider testing the homogeneity for proportions  in  independent binomial distributions
especially when data are sparse for large number of groups.
We provide broad aspects of our proposed tests such  as  theoretical studies, simulations and real data application.
We present the asymptotic null distributions and asymptotic powers for our proposed tests
and compare their performance with existing tests.
Our simulation studies show that none of tests dominate the others, however our proposed test and a few tests are expected to
control given sizes and obtain significant powers.
We also present a real example regarding safety concerns associated with Avandiar (rosiglitazone) in Nissen and Wolsky (2007).
\\
\noindent {\bf keywords} : Asymptotic distribution; homogeneity of proportions; sparse data
\end{abstract}
\section{Introduction}
\label{intro}
An important step in statistical meta-analysis is to carry out appropriate tests of homogeneity of the relevant {\it effect sizes}
before pooling of evidence or information across studies. While the familiar Cochran's (1954) chi-square goodness-of-fit
test is widely used in this context, it turns out that this test may perform poorly in terms of {\it not} maintaining Type I error
rate in many problems.
In particular,
this is indeed a serious drawback of Cochran's test for testing the homogeneity of several proportions in case of sparse data.
A recent meta-analysis (Nissen and Wolsky (2007), addressing the cardiovascular safety concerns associated with
(rosiglitazone), has received wide attention (Cai et al.(2010), Tian et al. (2009),
Shuster et al. (2007), Shuster
(2010), Stijnen et al. (2010)). Two difficulties seem to appear in this study: first, study sizes (N) are highly unequal, especially in
control arm, with
over $95 \%$ of the studies having sizes below 400 and two studies having sizes over 2500; second, event rate is extremely low,
especially for death end point,
with the maximum death rate in the treatment arm being $2 \%$, while in control arm, over $80 \%$ of the studies have zero events.
The original meta-analysis (Nissen and Wolski (2007))
was performed under fixed effects framework, as the diagnostic test based on Cochran's chi-square test failed to
reject homogeneity. However, with two large studies dominating the combined result, people agree random effects analysis is the
superior choice over fixed effects
(Shuster et al. (2007)). Moreover, the results for the fixed and random effects analyses are discordant. While different fixed effects
and random effects approaches are proposed, the problem of testing for homogeneity of effect sizes
is less familiar, and often not properly addressed. This is precisely the object of this paper, namely, a thorough
discussion of tests of homogeneity of proportions in case of sparse data situations.
 Recently, there are some studies on testing the equality of means when
the number of groups increases with fixed sample sizes in either  ANOVA (analysis of variance) or MANOVA (multivariate analysis of variance). For example, see
Bathke and Harrar (2008), Bathke and Lankowski (2005) and Boos and Brownie (1995).
Those studies have limitation in asymptotic results since they assume all samples sizes are equal, i.e., balanced design.
On the other hand, we actually emphasize the case that sample sizes are highly unbalanced and present more fluent asymptotic results
for a variety cases including unbalanced cases and small values of proportions in binomial distributions.

In this paper, we first point out that the classical chi-square test may fail in controlling a size when the number of groups is high and data are sparse.
We modify the classical chi-square test with providing asymptotic results.
Moreover, we propose two new tests for homogeneity of proportions when there are many groups with sparse count data.
Throughout this study, we present some theoretical conditions under which our proposed tests achieve the asymptotic normality while most of existing tests doesn't have rigorous investigation of asymptotic properties.

A formulation of the
testing problem for proportions is provided in Section 2 along with a review of the literature and suggestion for  {\it new} tests.
The necessary
asymptotic theory
to ease the application of the suggested test is developed.
Results of simulation studies are reported in Section 3 and  an application to the Nissen-Wolski (2007) data set is
made in Section 4. Concluding remark is presented in section 5.

\section{Testing the homogeneity of proportions with sparse data}
In this section,  we present a modification of a classical test which is Cochran's test and also propose two types of new tests.
Throughout this paper, our theoretical studies are based on triangular array
which is commonly used in asymptotic theories in high dimension.
See Park and Ghosh (2007) and Park (2009) for triangular array in binary data and  Greenshtein and Ritov (2004) for more general cases.
More specifically, let $\Theta^{(k)} = \{ (\pi_1^{(1)}, \pi_2^{(2)}, \ldots, \pi_k^{(k)}) :  0<\pi_{i}^{(k)}<1~~\mbox{for $1\leq i \leq k$}\}$
be the parameter space in which $\pi_i^{(k)}$s are allowed to be varying depending on $k$ as $k$ increases.
Additionally,  sample sizes $(n_1^{(1)}, \ldots, n_k^{(k)})$ also changes depending on $k$. However, for notational  simplicity, we suppress
superscript $k$ from $\pi_i^{(k)}$ and $n_i^{(k)}$.
The triangular array provides more flexible situations, for example all increasing sample sizes and all decreasing $\pi_i$s.
On the other hand, the asymptotic results in   Bathke and Lankowski (2005) and Boos and Brownie (1995)
are based on increasing $k$ but all sample sizes and $\pi_i$s are fixed. This set up provides somewhat limited results while we present
the asymptotic results on the triangular array.
Our results will include the asymptotic power functions of proposed tests while
existing studies do not provide them.
\subsection{Modification of Cochran's Test}
Suppose there are $k$ independent populations and the $i$th population
has $X_{i} \sim Binomial(n_i, \pi_i)$. Denote the total sample size and the weighted average of $\pi_i$'s by $N=\sumk n_i$ and $\bar
\pi = \frac{1}{N} \sumk n_i \pi_i$, respectively.  We are interested in testing the homogeneity of $\pi_{i}$'s from different
groups,
\begin{eqnarray} H_0 :  \pi_1 = \pi_2 = \cdots = \pi_k \equiv \pi (unknown).
\label{eqn:hypothesis}
\end{eqnarray}

To test the above hypothesis in (\ref{eqn:hypothesis}),
one familiar procedure is
Cochran's chi-square test in Cochran (1954),
namely $T_S$:

\begin{eqnarray} {T}_S = \sumk \frac{( X_i - n_i \hat {\bar \pi} )^2 }{n_i \hat {\bar \pi} (1-\hat {\bar \pi})}
\label{eqn:TS}
\end{eqnarray}

where $\hat \pi =
\frac{\sumk X_i }{\sumk n_i}$.  ${\cal T}_S$ uses
an approximate chi-square distribution with degrees of freedom $(k-1)$ under $H_0$.
The $H_0$ is rejected when ${T}_S > \chi^2_{1-\alpha, k-1}$
where $\chi^2_{1-\alpha, k-1}$ is the  $1-\alpha$ quantile of chi-square distribution with degrees of freedom
 $(k-1)$.
In particular, when $k$ is large,  $ \frac{T_S -k}{\sqrt{2k}}$ is approximated by a standard normal distribution under $H_0$.
Although Cochran's test for homogeneity is widely used, the approximation to the $\chi^2$ distribution
of $T_S$ or normal approximation may be poor when the sample
sizes within the groups are small or {when some counts in one of the two categories are low}.
This is partly because the test statistic becomes noticeably discontinuous and partly because its moments beyond the first
may be rather different from those of $\chi^2$.

We demonstrate that the asymptotic chi-square approximation to ${T}_S$ or
normal approximation based on $\frac{{T}_S-k}{\sqrt{2k}}$  may be very
poor when $k$ is large or $\pi_i$s are small compared to $n_i$s.
We provide the following theorem and propose a modified approximation to $T_S$ which is expected to
provide more accurate approximation.
Let us define
\begin{eqnarray}
T=\frac{{\cal T}_S - E({\cal T}_S)}{\sqrt{{\cal B}_k}}
\label{eqn:modTS}
\end{eqnarray}
where ${\cal T}_S = \sum_{i=1}^k \frac{(X_i - n_i \bar \pi)^2}{n_i \bar \pi (1-\bar \pi)}$,
${\cal B}_k \equiv  Var({\cal T}_S)  = \sum_{i=1}^k Var \left(\frac{(X_i - n_i \bar \pi)^2}{n_i \bar \pi (1-\bar \pi)} \right) \equiv  \sum_{i=1}^k B_i$ and
\begin{eqnarray*}
Var\left(\frac{X_i - n_i \bar \pi}{n_i \bar \pi (1- \bar \pi)}  \right)  = B_i &=& \frac{2\pi_i^2(1-\pi_i)^2}{\bar \pi^2 (1-\bar \pi)^2} + \frac{\pi_i(1-\pi_i)(1-6\pi_i(1-\pi_i))}{n_i\bar \pi^2(1-\bar \pi)^2} \\
&&+  \frac{3\pi_i(1-\pi_i)(1-2\pi_i)(\pi_i-\bar \pi)}{\bar \pi^2 (1-\bar \pi)^2} + \frac{4n_i \pi_i(1- \pi_i)(\pi_i-\bar \pi)^2}{\bar \pi^2 (1-\bar \pi)^2}, \\
E({\cal T}_S) &=& \sum_{i=1}^k \left( \frac{n_i(\pi_i -\bar \pi)^2 }{\bar \pi(1-\bar \pi)} + \frac{\pi_i(1-\pi_i)}{\bar \pi (1-\bar \pi)} \right).
\end{eqnarray*}
Note that ${\cal T}_S$ is not a statistic since it still includes the unknown parameter $\bar \pi = \sum_{i=1}^k \frac{n_i \pi_i}{N}$.
It will be shown later that $\bar \pi$ can be replaced by $\hat {\bar \pi} = \frac{1}{N} \sum_{i=1}^k n_i \hat \pi_i$ under $H_0$
since $\hat {\bar \pi}$ has the ratio consistency ($\frac{\hat {\bar \pi}}{\bar \pi} \rightarrow 1$ in probability)   under some mild conditions.
Define
\begin{eqnarray*}
{\cal B}_{0k} = \sum_{i=1}^k B_{0i} = \sum_{i=1}^k \left(2-\frac{6}{n_i} + \frac{1}{n_i \bar \pi(1-\bar \pi)} \right)
\end{eqnarray*}
and
\begin{eqnarray}
T_0 =  \frac{{\cal T}_S - k}{\sqrt{{\cal B}_{0k}}}
\label{eqn:modTS0}
\end{eqnarray}
which is the $T$ defined in (\ref{eqn:modTS})
under $H_0$ since $E({\cal  T}_S) =k$ and ${\cal B}_k = {\cal B}_{0k}$ under $H_0$.
The following theorem shows the asymptotic properties of $T_0$ in (\ref{eqn:modTS0}).
\begin{theorem}
For $\theta_i=\pi_i(1-\pi_i)$ and $\bar \theta = \bar \pi(1-\bar \pi)$, if $\frac{\sumk \left(\theta_i^4 + \frac{\theta_i}{n_i} \right)}{(\bar \pi (1-\bar \pi))^4 {\cal B}_k^2 } \rightarrow 0$ and $
\frac{\sumk n_i^2 \theta_i (\pi_i -\bar \pi)^4 (\theta_i + \frac{1}{n_i}) }{(\bar \pi (1-\bar \pi))^4{\cal B}_k^2} \rightarrow 0$ as  $k\rightarrow \infty$, then we have
\begin{eqnarray*}
P( T_0 > z_{1-\alpha}) - \bar \Phi \left( \frac{z_{1-\alpha}}{\sigma_k} -\mu_k \right) \rightarrow 0
\end{eqnarray*}
where
$\mu_k = \frac{E({\cal T}_S)-k}{\sqrt{{\cal B}_{k}}}$, $\sigma^2_k = \frac{{\cal B}_k}{{\cal B}_{0k}}$ and
$\bar \Phi (z) = 1-\Phi (z) =P(Z\geq z)    $ for a standard normal distribution $Z$.
\label{thm:modTS}
\end{theorem}
\begin{proof}
See Appendix.
\end{proof}
%
%

\bigskip
We propose to use a test which rejects the $H_0$ if
\begin{eqnarray}
T_{\chi} \equiv  \frac{{T}_S - k}{ \sqrt{\hat {\cal B}_{0k}}} > z_{1-\alpha}
\label{eqn:modifiedtest}
\end{eqnarray}
where $z_{1-\alpha}$ is the $1-\alpha$ quantile of a standard normal distribution,
 $\hat {\cal B}_{0k} =  \sum_{i=1}^k \left(2-\frac{6}{n_i} + \frac{1}{n_i \hat {\bar \pi}(1-\hat {\bar \pi})} \right)$
and   $\hat{\bar \pi} = \frac{\sum_{i=1}^k n_i \hat \pi_i}{N}$.

\bigskip
Using Theorem \ref{thm:modTS}, we obtain the following results which states that
our proposed modification of Cochran's test in (\ref{eqn:modifiedtest}) is the asymptotically size $\alpha$ test while
$\frac{{T}_S -k}{\sqrt{2k}}$ may fail in controlling a size $\alpha$
under some conditions.

\begin{corollary}
Under $H_0$ and the conditions in Theorem \ref{thm:modTS},
$T_{\chi}$ in (\ref{eqn:modifiedtest}) is asymptotically size $\alpha$ test.
A normal approximation to $\frac{{T}_S -k}{\sqrt{2k}}$ is not asymptotically  size $\alpha$ test
unless $\frac{{\cal B}_{0k}}{2k} \rightarrow 1$.
\label{cor:inconsistent}
\end{corollary}
\begin{proof}
We first show that $\hat {\bar \pi}/\bar \pi \rightarrow 1$ in probability.
Under $H_0$,  $\pi_i \equiv \pi$, we have $\sum_{i=1}^k n_i \hat \pi_i \sim Binomial(N,\pi)$. Using $ \sum_{i=1}^k n_i \pi_i = N\pi \rightarrow \infty$ under $H_0$, we have
\begin{eqnarray*}
E \left(\frac{\hat {\bar \pi}}{\bar \pi} -1 \right)^2 = \frac{1-\pi}{N \pi} \leq \frac{1}{N\pi} \rightarrow 0
\end{eqnarray*}
leading to  $\hat {\bar \pi}/\bar \pi \rightarrow 1$ in probability.
From this, we have $ \frac{\hat {\cal B}_{0k}}{{\cal B}_{0k} }
\rightarrow 1 $ in probability under $H_0$.
Furthermore, under $H_0$, since we have $ \frac{{\cal B}_{0k}}{{\cal B}_k}=1$ and $E({\cal T}_S)=k$, we obtain
$T_{\chi} - T = ( \sqrt{\frac{{\cal B}_{0k}}{ \hat {\cal B}_{0k} }} -1   )T  = o_p(1) O_p(1) = o_p(1)$ which means
$T_{\chi}$ and $T$ are asymptotically equivalent under the $H_0$.
Since $P_{H_0}(T > z_{1-\alpha} ) - \bar \Phi (z_{1-\alpha})  \rightarrow 0 $,
we have $P_{H_0}(T_{\chi} > z_{1-\alpha} ) - \alpha  \rightarrow 0$ which means
$T_{\chi}$  is the asymptotically size $\alpha$ test.
On the other hand, it is obvious that  $ \frac{{\cal T}_S -k}{\sqrt{2k}}$ doesn't have an asymptotic standard normality unless
${{\cal B}_{0k}}/(2k) \rightarrow 1$
 since  $  \frac{{\cal T}_S-k}{\sqrt{2k}}  =  \sqrt{\frac{\hat {\cal B}_{0k}}{2k}}   T_{\chi}$ under the $H_0$.
\end{proof}

\bigskip
Under $H_0$,  since ${\cal B}_{0k} = 2k + (\frac{1}{ \bar \pi(1-\bar \pi)} -6 ) \sum_{i=1}^k \frac{1}{n_i}$,
we expect $\frac{{\cal B}_{0k}}{2k}$ to converge to 1 when
$ (\frac{1}{\pi(1-\pi)} -6 ) \sum_{i=1}^k \frac{1}{n_i} =o(k)$ where $\pi_i= \bar \pi \equiv  \pi $ under $H_0$. This may happen
when  $\pi$ is bounded away from 0 and 1 and $n_i$s are large.
If  all $n_i$s are bounded by some constant, say $C$, and
     $|\frac{1}{\pi(1-\pi)} -6| \geq \delta >0$ (this can happen when $\pi <\epsilon_1$  or $\pi> 1-\epsilon_2$ for some $\epsilon_1>0$ and $\epsilon_2>0$),
then $\frac{{\cal B}_k}{2k}$ does not converge to 1.
 Even for $n_i$s are large,  if $\pi  \rightarrow 0$ fast enough, then $\frac{{\cal B}_{0k}}{2k}$ does not converge to 1.
For example, if $\pi=1/k$ and $n_i=k$ as $k \rightarrow \infty$,  then $\frac{{\cal B}_{0k}}{2k} \rightarrow 3/2$ which leads to
$\frac{{\cal T}_S-k}{\sqrt{2k}} \rightarrow N(0, \frac{3}{2})$ in distribution.
This implies that $P(\frac{{\cal T}_S -k}{\sqrt{2k}}  > z_{1-\alpha}) \rightarrow   1-\Phi (  \sqrt{\frac{2}{3}}  z_{1-\alpha} ) > \alpha$, so
the  test obtains a larger asymptotic size than a given nominal level.
To summarize, if either $\pi$ is small or $n_i$s are small,
we may not expect an accurate approximation to $\frac{{\cal T}_S-k}{\sqrt{2k}}$ based on normal approximation, so
the sparse binary data with small $n_i$s  and a large number of groups ($k$) needs to be handled more carefully.

\subsection{New Tests}
In addition to the modified Cochran's test $T_{\chi}$, we also propose new tests designed for sparse data when $k$ is large.
Similar to the asymptotic normality of $T_{\chi}$, it will be justified that
our proposed tests have the asymptotic normality  when $k \rightarrow \infty$ although $n_i$s are not required to increase.
Towards this end, we proceed as follows.
Let $||\mathb{\pi} - \mathb{\bar \pi}||^2_{{\bf n}} = \sum_{i=1}^k n_i (\pi_i - \bar \pi)^2$ which is
weighted $l_2$ distance from $\mathb{\pi}=(\pi_1,\pi_2,\ldots, \pi_k)$ to $\mathb{\bar \pi} = (\bar \pi, \bar \pi, \ldots, \bar \pi)$
where ${\bf n}=(n_1,\ldots, n_k)$.
The proposed test is based on measuring the $||\mathb{\pi} - \mathb{\bar \pi}||^2_{{\bf n}}$.
Since this is unknown, one needs to estimate the    $||\mathb{\pi} - \mathb{\bar \pi}||^2_{{\bf n}}$. One typical estimator
is a plug-in estimator such as $||\hat{\mathb{\pi}} - {\mathb{\hat {\bar \pi}}}||_{{\bf n}}$, however this
estimator may have a significant bias.
To illustrate this, note that
%
\begin{eqnarray*}
E||\hat{\mathb{\pi}} - {\mathb{\hat{ \bar\pi}}}||^2_{{\bf n}}  &=& \sumk \vari + \sumk \frac{n_i \vari}{N} - \frac{2}{N} \sumk n_i \vari + \sumk n_i (\pi_i - \bar \pi)^2 \\
 &=& \sumk c_i \vari + ||\mathb{\pi} - \mathb{\bar \pi}||^2_{{\bf n}}
\end{eqnarray*}
where $c_i =(1 - \frac{n_i}{N})$. This shows that $||\hat{\mathb{\pi}} - {\mathb{\hat{ \bar\pi}}}||^2_{{\bf n}}$
is an overestimate of $||\mathb{\pi} - \mathb{\bar \pi}||^2_{{\bf n}}$ by $\sumk c_i \vari$ which needs to be corrected.
Using $E\left[\frac{n_i}{n_i-1} \hatvari \right] =  \vari$ for $\hat \pi_i = \frac{x_i}{n_i}$,
we define $d_i=\frac{n_i c_i}{n_i-1}$ and
\begin{eqnarray}
T  = \sum_{i=1}^n n_i (\hat \pi_i - \hat {\bar \pi})^2 - \sumk d_i \hatvari  \equiv   ||\hat{\mathb{\pi}} - {\mathb{\hat{ \bar\pi}}}||^2_{{\bf n}} - \sumk
d_i \hatvari
\label{eqn:T}
\end{eqnarray}
which is an unbiased estimator of $||\mathb{\pi} - \mathb{\bar \pi}||^2_{{\bf n}}$.
This implies $E(T)= ||\mathb{\pi} - \mathb{\bar \pi}||^2_{{\bf n}} \geq 0$ and "=" holds only when $H_0$ is true.
Therefore, it is natural to consider large values of $T$ as an evidence supporting $H_1$, and we thus propose
a one-sided (upper) rejection region based on $T$ for testing $H_0$.
Our proposed test statistics are based on $T$ of which the asymptotic distribution is normal distribution under some conditions.

We  derive the asymptotic normality of a standardized version of $T$ under some regularity conditions.
Let us decompose $T$  into two components, say $T_1$ and $T_2$:
\begin{eqnarray}
 T &=&  \sumk n_i(\hat \pi_i - \pi_i + \pi_i - \bar \pi  + \bar \pi - \hat{\bar \pi} )^2 - \sumk d_i \hatvari   \nonumber \\
&=&
\underbrace{\sumk \left\{ n_i (\hat \pi_i - \pi_i)^2 - d_i \hatvari  + 2 n_i (\hat \pi_i - \pi_i ) (\pi_i-\bar \pi) +n_i (\pi_i -\bar \pi)^2 \right\}}_{T_1} \\
&&- \underbrace{ N (\hat {\bar \pi} - \bar \pi )^2 }_{T2}   \label{eqn:T1T2}
 \end{eqnarray}
 where $T_1 \equiv \sumk T_{1i}$ for $T_{1i} = n_i (\hat \pi_i - \pi_i)^2 - d_i \hatvari +2 n_i (\hat \pi_i - \pi_i ) (\pi_i-\bar \pi)+n_i (\pi_i -\bar \pi)^2$.
%
%
%
%
%
%
To prove the asymptotic normality of the proposed test, we need some preliminary results stated below in Lemmas \ref{lemma:1}, \ref{lemma:varianceT1} and \ref{lemma:vart1}, and
show the ratio consistency of proposed estimators of $Var(T_1)$ in Lemma \ref{lemma:ratio}.

\begin{lemma}
Let $\theta_i= \vari$.  When  $X_i \sim Binomial (n_i, \pi_i)$ and $\hat \pi_i = \frac{X_i}{n_i}$,  we have
\begin{eqnarray*}
E[(\hat \pi_i - \pi_i)^3] & =& \frac{(1-2\pi_i)\theta_i}{n_i^2},~~  E[(\hat \pi_i - \pi_i)^4] = \frac{3\theta_i^2}{n_i^2} + \frac{(1-6\theta_i)\theta_i}{n_i^3}\\
E[\hatvari] &=& \frac{n_i-1}{n_i} \theta_i,  \\
\pi_i^l &=& E\left[\frac{n_i^l}{\prod_{j=0}^{l-1} \left(n_i -j \right)} \prod_{j=0}^{l-1} \left(\hat \pi_i -\frac{j}{n_i} \right) \right]  ,
~~~~\mbox{for $n_i \geq l$ and $l=1,2,3,4$}.
\end{eqnarray*}
\label{lemma:1}
\end{lemma}
\begin{proof}
   The first three results are easily derived by some computations.
For the last result, note that
when $X_i\sim Binomial(n_i,\pi_i)$,   $E[X_i(X_i-1)\cdots (X_i-l+1)] =  n_i(n_i-1)\cdots (n_i-l+1) \pi_i^l$.
Let $X=\sumk \sum_{j=1}^{n_i} X_{ij} \sim Binomia\l(N, \pi)$, then we have the above unbiased estimators under $H_0$ using
 $\hat \pi =  \frac{X}{N} = \frac{1}{N} \sumk n_i \hat \pi_i$.
 \end{proof}

%
\bigskip
We now derive the asymptotic null distribution of $\frac{T_1}{\sqrt{Var(T_1)}}$ and propose an unbiased estimator of ${Var(T_1)}$
which has the ratio consistency property.
We first compute $Var(T_1)$ and then propose an estimator $\widehat{Var(T_1)}$.

\begin{lemma}
\label{lemma:varianceT1}
The variance of $T_1$, $Var(T_1)$, is
\begin{eqnarray}
Var(T_1) =\sumk {\cal A}_{1i}  \theta_i^2 + \sumk {\cal A}_{2i} \theta_i + 4 \sum_{i=1}^k n_i(\pi_i - \bar \pi)^2 \theta_i
+\frac{4}{N}\sumk n_i(\pi_i - \bar \pi)(1-2\pi_i)\theta_i
\label{eqn:varianceT1}
\end{eqnarray}
where $ {\cal A}_{1i}= \left( 2-\frac{6}{n_i} - \frac{d_i^2}{n_i} + \frac{8d_i^2}{n_i^2}-\frac{6d_i^2}{n_i^3} + 12 d_i \frac{n_i-1}{n_i^2}  \right)$
and ${\cal A}_{2i}=  \frac{ n_i}{N^2}$
for $d_i = \frac{n_i}{n_i-1} \left(1-\frac{n_i}{N} \right)$ .
\end{lemma}
\begin{proof}
See Appendix.
\end{proof}

Under the $H_0$ ( $\pi_i = \pi$ for all $1\leq i \leq k$),
 the third and fourth terms including $\pi_i-\bar \pi$ in (\ref{eqn:varianceT1})
 are 0 and therefore we obtain the $Var(T_1)$ under $H_0$ as follows;
\begin{eqnarray}
Var_{H_0}(T_1) &\equiv & {\cal V}_1 = \sum_{i=1}^k  \left\{  {\cal A}_{1i} \theta_i^2 +  {\cal A}_{2i} \theta_i  \right\}  \label{eqn:estvar1}  \\
&=& {\cal V}_{1*}  =   (\pi(1-\pi))^2 \sum_{i=1}^k {\cal A}_{1i}+ \pi(1-\pi) \sum_{i=1}^k {\cal A}_{2i} . \label{eqn:estvar2}
\end{eqnarray}

${\cal V}_1$ in (\ref{eqn:estvar1}) and ${\cal V}_{1*}$ in (\ref{eqn:estvar2})
are equivalent under the $H_0$, however
the estimators may be different depending on whether $\theta_i$s are estimated individually from $x_i$
or the common value $\pi$ is estimated in ${\cal V}_{1*}$ by the pooled estimator  $\hat \pi$.
 We shall consider consider these two approaches for estimating ${\cal V}_1$ and ${\cal V}_{1*}$.

\bigskip
First, we demonstrate the estimator for ${\cal V}_1$ in (\ref{eqn:estvar1}).
$ {\cal V}_{1i}  \equiv {\cal A}_{1i} \theta_i^2 + {\cal A}_{2i} \theta_i$
is a 4th degree polynomial in $\pi_i$,
in other words,
${\cal V}_{1i} = a_{1i} \pi_i + a_{2i} \pi_i^2 + a_{3i} \pi_3^3 + a_{14} \pi_i^4$ where
$a_{ij}$'s depend only on $N$ and $n_i$.
As an estimator of ${\cal V}_1 = \sumk (a_{1i} \pi_i + a_{2i} \pi_i^2 + a_{3i} \pi_i^3 + a_{4i} \pi_i^4)$, we consider
unbiased estimators of $\pi_i$, $\pi_i^2$, $\pi_i^3$ and $\pi_i^4$.
Let $\eta_{li}= \pi_{i}^l$, $l=1,2,3,4$, then unbiased estimators of
$\eta_{li}$, say $\hat \eta_{li}$, are obtained directly from Lemma \ref{lemma:1}, leading to the first estimator of ${\cal V}_{1}$,
as
\begin{eqnarray} \hat{\cal{V}}_1 = \sumk \sum_{l=1}^{4} a_{li} \hat \eta_{li}
\label{eqn:Var1}
\end{eqnarray}
where $ \hat \eta_{li} = \frac{n_i^l}{\prod_{j=1}^{l-1} (n_i-j)} \prod_{j=0}^{l-1} \left(\hat \pi_i - \frac{j}{n_i} \right)$ for
$l=1,2,3,4$ from Lemma \ref{lemma:1} and
\begin{eqnarray*}
a_{1i}= {\cal A}_{2i},~~
a_{2i} =  {\cal A}_{1i}-{\cal A}_{2i},~~
a_{3i} = -2{\cal A}_{1i},~~
a_{4i} &=&  {\cal A}_{1i}.
\end{eqnarray*}

The second estimator is  based on estimating  ${\cal V}_{1*}$ in (\ref{eqn:estvar2}).
Since all $\pi_i=\pi$ under $H_0$,
we can write ${\cal V}_{1*}= \sum_{i=1}^{k} \sum_{l=1}^4 a_{li} \pi_i^l
= \sumk \sum_{l=1}^4 a_{li} \pi^l$, and use an unbiased estimator of $\pi^l$
 using  $\sum_{i=1}^k x_i \sim Binomial(N, \pi)$ from  Lemma \ref{lemma:1}.
 This leads to the estimator of ${\cal V}_{1*}$ under $H_0$ which is
\begin{eqnarray}
\hat{{\cal V}}_{1*} =  \sumk \sum_{l=1}^4 a_{li} \hat \eta_{l}.
\label{eqn:Var2}
\end{eqnarray}
where $\hat \eta_{l} = \frac{N^l}{\prod_{j=0}^l(N-j)} \prod_{j=0}^{l-1} \left(\hat \pi - \frac{j}{N} \right)$
and $\hat \pi = \frac{1}{N} \sumk n_i \hat \pi_i$, as used earlier.

\begin{remark}
Note that $\hat {\cal V}_1$ is an unbiased estimator of ${\cal V}_1$ regardless of $H_0$ and $H_1$.
On the other hand, $\hat {\cal V}_{1*}$ is an unbiased estimator of ${\cal V}_{1*}$
only under the $H_0$ since we use the binomial distribution of the pooled data $\sum_{i=1}^k x_i$ and use the Lemma \ref{lemma:1}.
\end{remark}

For sequences of $a_n (>0)$ and $b_n(>0)$,   let us define  $a_n \asymp b_n $ if  $ 0< \liminf \frac{a_n}{b_n}  \leq \limsup \frac{a_n}{b_n}  < \infty $.
The following lemmas will be used in the asymptotic normality of the proposed test.

\begin{lemma}  Suppose  $n_i \geq 2$ for $1\leq i \leq k$. Then,
\begin{enumerate}
\item  we have ${\cal V}_1 \asymp  \sum_{i=1}^k \theta_i^2  +  \frac{1}{N^2}  \sum_{i=1}^k  n_i \theta_i$.
In particular, if $0<c\leq \pi_i \leq 1-c <1$ for all $i$ and some constant $c$,  we have ${\cal V}_1\asymp k$.
\item we have
\begin{eqnarray}
  \sum_{i=1}^k {\cal A}_{1i} \theta_i^2  \leq  Var(T_1) \leq K( {\cal V}_1  +   ||{\mathb \pi} - \bar {\mathb \pi} ||_{{\bf n} \theta}^2)
  \end{eqnarray}
 for some constant $K>0$
 where  $||{\mathb \pi} - \bar {\mathb \pi} ||_{{\bf n} \theta}^2 = \sum_{i=1}^k n_i (\pi_i - \bar \pi)^2 \theta_i$.
If   $ |\pi_i - \bar \pi| \geq \frac{1+\epsilon}{N}$ for some $\epsilon >0$ and $1\leq i\leq k$, we have
\begin{eqnarray}
Var(T_1) \asymp {\cal V}_1 + ||{\mathb \pi} - \bar {\mathb \pi} ||_{{\bf n} \theta}^2. \label{eqn:asympVar}
\end{eqnarray}
\end{enumerate}
\label{lemma:vart1}
\end{lemma}

\begin{proof}
See Appendix.
\end{proof}


We provide another lemma which plays a crucial role in the proof of the main result.
As mentioned, we have two types of variances such as  ${\cal V}_1$ in (\ref{eqn:estvar1}) and ${\cal V}_{1*}$ in (\ref{eqn:estvar2})
and their estimators $\widehat{{\cal V}_1}$ and
$\widehat{{\cal V}}_{1*}$.
For  $T_1$ in (\ref{eqn:T1T2}), we consider two types of standard deviations based on  $Var(T_1)$ and $Var(T_1)_*$.

The following lemma provides upper bounds of $n^4 E(\hat \pi - \pi)^8$  and  $E(\hat \pi (1-\hat \pi))^4$
which are needed in our proof for our mail results.

\begin{lemma}
 If $X \sim Binomial(n,\pi)$, $\hat \pi = \frac{X}{n}$ and $\hat \eta_l$ is the unbiased estimator of $\pi^l$ defined in Lemma \ref{lemma:1}, then
 we have, for $\theta \equiv \pi(1-\pi)$,
\begin{eqnarray}
n^4 E(\hat \pi - \pi)^8 &\leq& C \min \left\{  \theta^4, \frac{\theta}{n} \right \} \nonumber \\
E(\hat \pi (1-\hat \pi))^4 &\leq & C'\min\left\{ \theta^4, \frac{\theta}{n^3} \right \} \nonumber\\
E \hat \pi^l &=&  \pi^l + O \left(\frac{\pi}{n^{l-1}} + \frac{\pi^{l-1}}{n}  \right)~~\mbox{for $l \geq 2$} \label{eqn:hatpipowerk} \\
E (\hat \pi^l - \pi^l)^2 &=&  O\left(\frac{\pi^{2l-1}}{n} + \frac{\pi}{n^{2l-1}} \right)~~\mbox{for $l\geq 2$} \nonumber \\
E (\hat \eta_l - \pi^l)^2 &=&  O\left(\frac{\pi^{2l-1}}{n} + \frac{\pi}{n^{2l-1}} \right)~~\mbox{for $l\geq 2$} \nonumber \\
\end{eqnarray}
where $C$ and $C'$ are universal constants which do not depend on $\pi$ and $n$.
\label{lemma:lemma4}
\end{lemma}
\begin{proof} See Appendix.  \end{proof}

\begin{remark}
 It should be noted that the bounds in Lemma \ref{lemma:lemma4}
depend on the behavior of $\theta =\pi(1-\pi)$ and the sample size $n$ in binomial distribution.
In the classical asymptotic theory for a fixed value of $\pi$,
if $\pi$ is bounded away from $0$ and $1$ and $n$ is large,
then $\theta^4$ dominates $\frac{\theta}{n}$ (or $\frac{\theta}{n^3}$). However,
$n$ is not large and $\pi$ is close to 0 or 1, then $\frac{\theta}{n}$ (or $\frac{\theta}{n^3}$)
is a tighter bound of   $n^4 E(\hat \pi - \pi)^8$ (or $E(\hat \pi (1-\hat \pi))^4$) than $\theta^4$.
\end{remark}

The following lemma shows that   $\hat {\cal V}_1$ and $\hat {\cal V}_{1*}$ have the ratio consistency under
some conditions.
\begin{lemma}
For $\tilde \theta = \bar \pi(1-\bar \pi)$, $\bar \pi =  \frac{1}{N} \sumk n_i \pi_i$ and $\pi_i \leq \delta <1$ for some $0<\delta<1$,
we have the followings;
\begin{enumerate}
\item if  $\frac{\sum_{i=1}^k \left(\frac{\theta_i^3}{n_i}+ \frac{\theta_i}{n_i^3}\right)}
{(\sum_{i=1}^k (\theta_i^2  + \frac{1}{N^2} \frac{\theta_i}{n_i} ) )^2} \rightarrow 0$ as $k\rightarrow 0$,
$\frac{\hat{\cal V}_1}{{\cal V}_1} \rightarrow 1$ in probability.
\item if  $\frac{ (\tilde \theta)^3 \sumk \frac{1}{n_i}  + \tilde \theta \sum_{i=1}^k \frac{1}{n_i^3}}
{ \left( k (\tilde \theta)^2+  \frac{\tilde \theta}{N^2}  \sum_{i=1}^k \frac{1}{n_i} \right)^2  }  \rightarrow 0 $,    $\frac{\hat{\cal V}_{1*}}{{\cal V}_{1*}} \rightarrow 1$ in probability.
\end{enumerate}
\label{lemma:ratio} \end{lemma}

\begin{proof}
See Appendix.
\end{proof}

\begin{remark}
\label{remark:dense}
Lemma \ref{lemma:ratio} includes the condition
$\pi_i \leq \delta <1$ which avoids dense case that the majority of observations are 1.
Since our study focuses on sparse case, it is realistic to exclude  $\pi_i$s which are very close to 1.
When data are dense, the homogeneity test of $\pi_i$ can be done through  testing $\pi_i^* \equiv 1-\pi$ and
$x_{ij}^*=1-x_{ij}$.
\end{remark}

\begin{remark}
\label{remark:mleV}
As an estimator of $\pi_i^l$ or $\pi^l$ for $l=1,2,3,4$,
we used unbiased estimators of them.
Instead of unbiased estimators, we may consider simply MLE, $ (\hat \pi_i)^l$ or $ (\hat {\pi})^l$ for $l=1,2,3,4$.
For the first type estimator  $\hat {\cal V}_1$, when sample sizes $n_i$ are not large, unbiased estimators and MLE are different. Especially,
if all $n_i$s are small and $k$ is large, then such small differences are accumulated so the behavior of estimators for variance
are expected to be significantly different. This will be demonstrated in our simulation studies.
On the other hand, for  $\hat {\cal V}_{1*}$, unbiased estimators and MLEs for $ (\pi)^l$ under $H_0$
behave almost same way even for small $n_i$ since
the total sample size $N=\sum_{i=1}^k n_i$ is large due to large $k$.
The estimator of ${\cal V}_{1}$ based on $\hat \pi_i$, namely $\hat {\cal V}_{1}^{mle}$ has the larger variance
\begin{eqnarray*}
E(\hat {\cal V}_1^{mle} - {\cal V}_1)^2 \asymp \sumk \left(\frac{\theta_i^3}{n_i} + \frac{\theta_i}{n_i^3}\right) + \sum_{i\neq j} \frac{\theta_i \theta_j}{n_in_j}
\end{eqnarray*}
while  $E(\hat {\cal V}_1 - {\cal V}_1)^2 \asymp   \sumk \left(\frac{\theta_i^3}{n_i} + \frac{\theta_i}{n_i^3}\right)$.
Similarly, we can also define ${\hat {\cal V}}_{1*}^{mle}$ based on the $\hat {\pi} = \frac{\sum_{i=1}^k x_{i}}{N}$.
Even with the given condition $ \sumk \left(\frac{\theta_i^3}{n_i}+\frac{\theta_i}{n_i^3}\right)/(\sumk \theta_i^2 + \frac{1}{N^2} \sumk \frac{\theta_i}{n_i})^2 =o(1)$,
$\hat {\cal V}_{1}^{mle}$ may not be a ratio consistent estimator due to the additional variation from biased estimation of $\pi_i^l$ for $l=2,3,4$.    We present simulation studies comparing tests with $\hat {\cal V}_1$ and $\hat {\cal V}_1^{mle}$ later.
\end{remark}

In Lemma \ref{lemma:ratio}, we present ratio consistency of  $\hat {\cal V}_1$ and $\hat {\cal V}_{1*}$
under some conditions.
Both conditions avoids too small $\pi_i$s compared to $n_i$s among $k$ groups.  It is obvious that the conditions are satisfied
if all $\pi_i$s are uniformly bounded away from 0 and 1. In general, however, the conditions allow small $\pi_i$s which may converge to zero at some rate satisfying presented conditions on $\theta_i$s in lemmas and theorems.

Under $H_0$,  we have two different estimators,  $\hat {\cal V}_1$ and $\hat {\cal V}_{1*}$
and their corresponding test statistics, namely  $T_{new1}$ and $T_{new2}$ respectively:
\begin{eqnarray*}
T_{new1} = \frac{T}{\sqrt{\hat{\cal V}_1 }},~~~~T_{new2} = \frac{T}{\sqrt{\hat{\cal V}_{1*}   }}.
\end{eqnarray*}

The following theorem shows that the proposed tests, $T_{new1}$ and $T_{new2}$, are asymptotically size $\alpha$ tests.

\begin{theorem} Under $H_0 : \pi_i \equiv \pi$ for all $1\leq i \leq k$,
if the condition in Lemma \ref{lemma:ratio} holds and
$\frac{ \sum_{i=1}^k \frac{1}{n_i}  }{k\theta^3} \rightarrow 0$
for $\theta = \pi(1-\pi)$ under $H_0$,
then  $T_{new1}\rightarrow N(0,1)$ in distribution and   $T_{new2}  \rightarrow N(0,1)$ in distribution as $k \rightarrow \infty$.
\label{cor:size}
\end{theorem}

\begin{proof}
See Appendix.
\end{proof}

\begin{remark}
The condition in Lemma \ref{lemma:ratio} under the $H_0$
is $\frac{\theta^3 \sumk \frac{1}{n_i} + \theta \sumk \frac{1}{n_i^3}}
{ \left(k\theta^2 + \frac{\theta}{N^2} \sum_{i=1}^k \frac{1}{n_i} \right)^2 } =o(1)$.
This condition includes a variety of situations such as small values of $\pi$
as well as small sample sizes. Furthermore, inhomogeneous sample sizes
are also included.
For example, when the sample sizes are bounded, we have
$\sumk \frac{1}{n_i} \asymp k$ and $\sumk \frac{1}{n_i^3} \asymp k$ leading to
  $\frac{\theta^3 \sumk \frac{1}{n_i} + \theta \sumk \frac{1}{n_i^3}}
{ \left(k\theta^2 + \frac{\theta}{N^2} \sum_{i=1}^k \frac{1}{n_i} \right)^2 }  \leq \frac{1}{k\theta^3}$ which converges to 0
when $k \theta^3 \rightarrow \infty$. This happens when $\pi = k^{\epsilon -1/3}$ for $0< \epsilon <1/3$
which is allowed to converge to 0.
Another case is that sample sizes are highly unbalanced.
For example, we have
$ n_i \asymp i^{\alpha}$ for $ \alpha >1$
which implies $\sum_{i=1}^{\infty} \frac{1}{n_i} < \infty$
and $\sum_{k=1}^{\infty} \frac{1}{n_i^3} <\infty$. Therefore the condition is
$  \frac{\theta^3 \sumk \frac{1}{n_i} + \theta \sumk \frac{1}{n_i^3}}
{ \left(k\theta^2 + \theta \sum_{i=1}^k \frac{1}{n_i} \right)^2 }
\asymp  \frac{ \theta^3  +  \theta  }{ (k \theta^2  +  \theta)^2 } \leq   \frac{ \theta^3 +  \theta  }{k^2 \theta^4 }
= \frac{1}{k^2 \theta} + \frac{1}{k^2 \theta^3} \rightarrow 0 $
 if $ \pi \asymp  k ^{\epsilon } $ for $ -\frac{2}{3} < \epsilon < 0$.
 In this case, the sample size $n_i$ diverges as $i \rightarrow \infty$, so
 sample sizes are highly unbalanced.  For the asymptotic normality, additional condition  $\sumk \frac{1}{n_i}/(k\theta^3) \rightarrow 0$  in Theorem \ref{cor:size} is satisfied  for $ -\frac{1}{3} < \epsilon <0$.
 \end{remark}

From Theorem \ref{cor:size},  we reject the $H_0$ if
\begin{eqnarray*}
 T_{new1} (\mbox{or}~~ T_{new2})  > z_{1-\alpha}
 \end{eqnarray*}
where $z_{1-\alpha}$ is $(1-\alpha)$ quantile of a standard normal distribution.
As explained in section 2.2, note that the rejection region is one-sided
since  we have $E(T) \geq 0$ implying that large values of tests support the alternative hypothesis.

\bigskip
 Although they have the same asymptotic null distribution,  their power functions are different due to the different behavior of
$\hat {\cal V}_1$ and $\hat {\cal V}_{1*}$ under $H_1$.
In general, it is not necessary to have the asymptotic normality under the $H_1$, however
to compare the powers analytically, one may expect asymptotic power functions to be more specific.

\bigskip
The following lemma states the asymptotic normality of $T/\sqrt{{Var(T_1)}}$
where $Var(T_1)$ is in (\ref{eqn:varianceT1}) in Lemma \ref{lemma:varianceT1}.
In the following asymptotic results, it is worth mentioning that we put some conditions on $\theta_i$s so that
 they do not approach to 0 too fast.

\begin{theorem}
If $(i)$ $ |\pi_i -\bar \pi| \geq \frac{1+\epsilon}{N}$ for $1\leq i \leq k$, $(ii)$
$\frac{\sumk (\theta_i^4 + \frac{\theta_i}{n_i})}{\left(\sumk \theta_i^4 + \frac{1}{N^2} \sumk \frac{\theta_i}{N}  \right)^2   } \rightarrow 0$ and
$(iii)$ $\frac{ \max_i  (\pi_i - \bar \pi)^2 (n_i\theta_i +1)}
{  {\cal V}_1 + ||{\mathb \pi} - \bar {\mathb \pi} ||^2_{\theta {\bf n}}} \rightarrow 0$
where $||{\mathb \pi} - \bar {\mathb \pi} ||^2_{\theta{\bf n}}  = \sum_{i=1}^k n_i (\pi_i-\bar \pi)^2 \theta_i$, then
\begin{eqnarray*}
\frac{T -\sum_{i=1}^k n_i (\pi_i - \bar \pi)^2}{\sqrt{{Var(T_1)}}} \rightarrow N(0,1) ~~\mbox{in distribution}
\end{eqnarray*}
where $Var(T_1)$ is defined in (\ref{eqn:varianceT1}).
\label{thm:normality}
\end{theorem}
\begin{proof}
See Appendix.
\end{proof}

Using Theorem \ref{thm:normality},  we obtain
the asymptotic power of the proposed tests. We state this in the following corollary.
\begin{corollary} Under the assumptions in Lemma \ref{lemma:ratio} and Theorem \ref{thm:normality},
the powers of $T_{new1}$ and $T_{new2}$ are
\begin{eqnarray*}
P(T_{new1} > z_{1-\alpha}) &-& \bar \Phi \left( \frac{\sqrt{{\cal V}_1}}{\sqrt{{Var(T_1)}}} z_{1-\alpha} - \frac{\sum_{i=1}^k n_i(\pi_i-\bar \pi)^2}{\sqrt{Var(T_1)}}  \right) \rightarrow 0
\end{eqnarray*}
and
\begin{eqnarray*}
P(T_{new2} > z_{1-\alpha}) &-& \bar \Phi \left( \frac{\sqrt{{\cal V}_{1*}}}{\sqrt{Var(T_1)}}  z_{1-\alpha}- \frac{\sum_{i=1}^k n_i(\pi_i-\bar \pi)^2}{\sqrt{Var(T_1)}}  \right) \rightarrow 0
\end{eqnarray*}
where  $\bar \Phi (x) = 1-\Phi (x) = P(Z > x)$ for a standard normal random variable $Z$  and  $Var(T_1)$  defined in (\ref{eqn:varianceT1}).
\label{cor:cor1}
\end{corollary}

\subsection{Comparison of Powers}
In the previous section, we present the asymptotic power of  tests,  $T_{new1}$ and $T_{new2}$.
Currently, it doesn't look straightforward to tell one test
is uniformly better than the others.
However,  one may consider some specific scenario and compare different tests under those scenario
which may help to understand the properties of tests in a better way.
Asymptotic powers depend on the configurations of $(\pi_i's)$, $(n_i's)$ and $k$.
It is not possible to consider all configurations, however
what we want to show through simulations is that neither of $T_{new1}$ and $T_{new2}$ dominates the other.

Let $\beta(T)$ be the asymptotic power of a test statistic  $\lim_{k\rightarrow \infty} P(T >z_{1-\alpha})$  where $T$ is one of $T_{\chi}$, $T_{new1}$ and $T_{new2}$.
\begin{theorem}
\label{eqn:powercomparison}
\begin{enumerate}
\item If sample sizes $n_1=\ldots=n_k\equiv n$ and  $\max_{1\leq i \leq k} \pi_i < \frac{1}{2}-\frac{1}{\sqrt{3}}$,
then
\begin{eqnarray*}
\lim_{k\rightarrow \infty}(\beta(T_{new2}) - \beta(T_{new1}) ) \geq 0.
\end{eqnarray*}
If $n_i=n$  for all $1\leq i \leq k$  and $ n\bar \pi (1-\bar \pi) \rightarrow \infty$, then
\begin{eqnarray*}
\lim_{k \rightarrow \infty} (\beta({T_{new2}}) - \beta(T_{\chi}) ) &\geq& 0.
\end{eqnarray*}

\item Suppose $\pi_i =\pi=k^{-\gamma}$ for $1\leq i\leq k-1$ and $\pi_k = k^{-\gamma}+\delta$ for $0<\gamma <1$
as well as $n_i =n$ for $1\leq i\leq k-1$,  and $n_k =  [nk^{\alpha}]$ for $0<\alpha < 1$
where $[x]$ is the greatest integer which does not exceed $x$. Then, if $n \rightarrow \infty$,
\begin{enumerate}
\item for $\{(\alpha,\gamma) : 0<\alpha<1, 0<\gamma<1, 0<\alpha + \gamma <1,   0<\gamma \leq \frac{1}{2}\}$, then  $\lim_k (\beta(T_{new_1}) -\beta(T_{new2}))=0$.
\item for $\{ (\alpha, \gamma) :0<\alpha<1, 0<\gamma<1, \alpha + \gamma >1,  \alpha > \frac{1}{2} \} $, then $\lim_k (\beta(T_{new_1}) -\beta(T_{new2})) >0$.
\end{enumerate}

\item  Suppose $\pi_1 =  k^{-\gamma}+\delta$  and  $n_1 =n \rightarrow \infty$  and
$\pi_i = k^{-\gamma}$ and $n_i = [n k^{\alpha}]$ for $2 \leq i \leq n$.
For $0<\gamma <1$ and $0<\alpha <1$,
if $0<\gamma <1/2$ and $k^{1-\alpha-\gamma} =o(n)$, then
 \begin{eqnarray}
\lim_{k \rightarrow \infty} (\beta(T_{new_2}) -\beta(T_{new_1})) >0.
\end{eqnarray}

\end{enumerate}
\label{thm:powercomparison}
\end{theorem}

\begin{proof}
See Appendix.
\end{proof}

From Theorem \ref{thm:powercomparison}, we conjecture that   $T_{new2}$ has better powers than others
when sample sizes are homogeneous or similar to each other.
For inhomogeneous sample sizes,   $T_{new1}$ and $T_{new2}$ have different performances from the cases of
2 and 3 in Theorem \ref{thm:powercomparison}. We show numerical studies reflecting these cases later.

\bigskip
Although we compare the powers of the proposed tests under some local alternative,
it is interesting to see different scenario and compare powers.
Instead of an analytical approach, we present numerical studies as follows.
Since the asymptotic powers of $T_{new1}$ and $T_{new2}$ depend on the behavior of
${\cal V}_1$ and ${\cal V}_{1*}$, we compare those two variances under a variety of situations.
If   ${\cal V}_{1*}> {\cal V}_1$, then
$T_{new1}$ is more powerful than $T_{new2}$;otherwise, we have an opposite result.
Although we compared the powers of tests in this paper in Theorem \ref{thm:powercomparison},
there are numerous additional situations which are not covered analytically.
We provide some additional situations from numerical studies here.
We take $k=100$ and we generate sample sizes $n_i \sim \{20,21,\ldots,200 \}$ uniformly.  The left panel is for $\pi_i \sim U(0.01, 0.2)$ and
the left panel is for $\pi \sim U(0.01,0.5)$ where $U(a,b)$ is the uniform distribution
in $(a,b)$.   We consider $1,000$ different configurations of $(n_i, \pi_i)_{1\leq i \leq 100}$ for each panel.
We see that $Var(T_1)$ and $Var(T_1)_*$ have different behavior when $\pi_i$s are
generated different ways. If $\pi_i$s are widely spread out, then $Var(T_1)_*$ is larger, otherwise  $Var(T_1)$ seems to be larger from our simulations.

\begin{figure}[ht]
\begin{center}
\caption{Comparison of $Var(T_1)$ and $Var(T_1)_*$}
\includegraphics[scale=0.3]{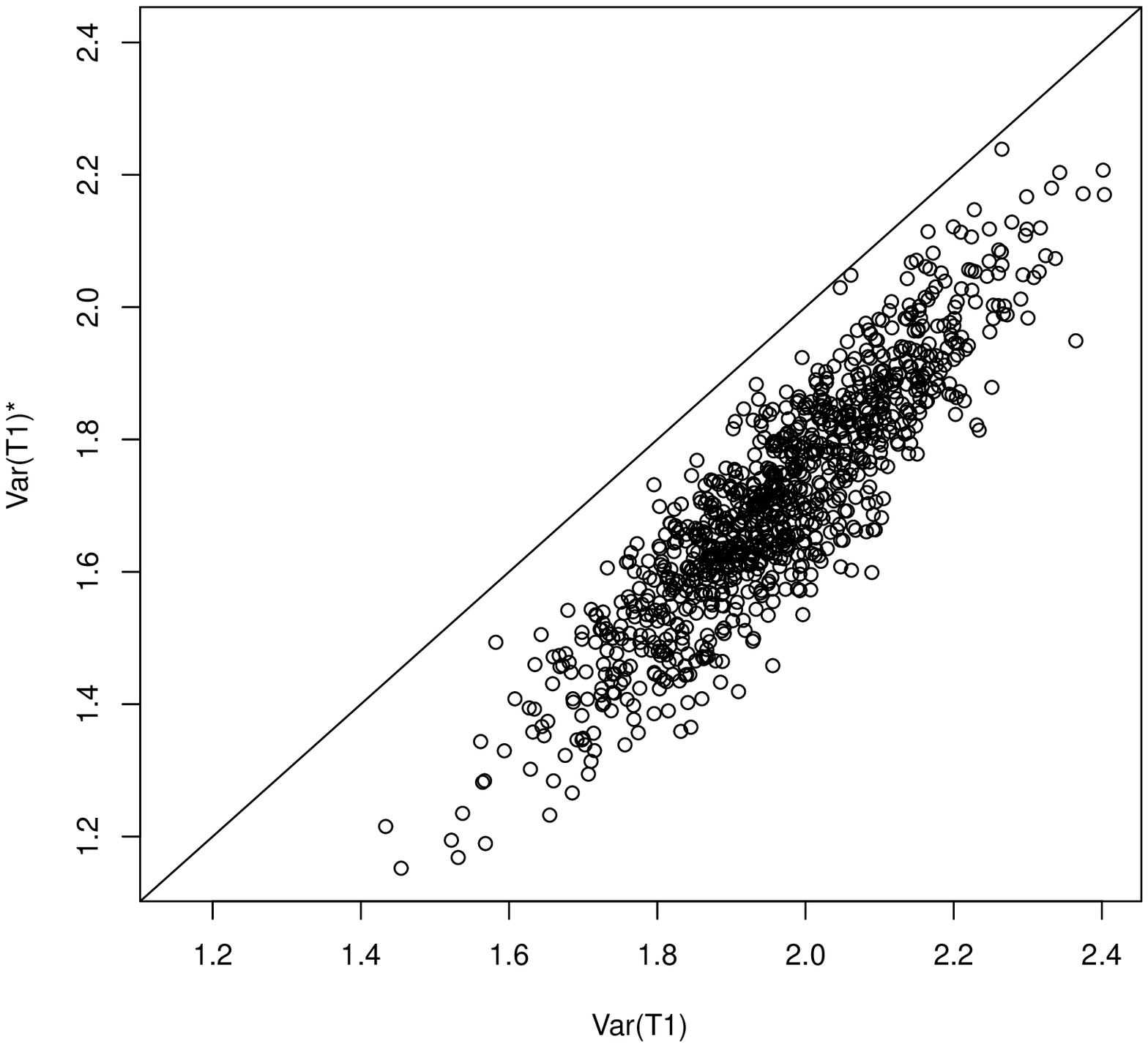}
\includegraphics[scale=0.3]{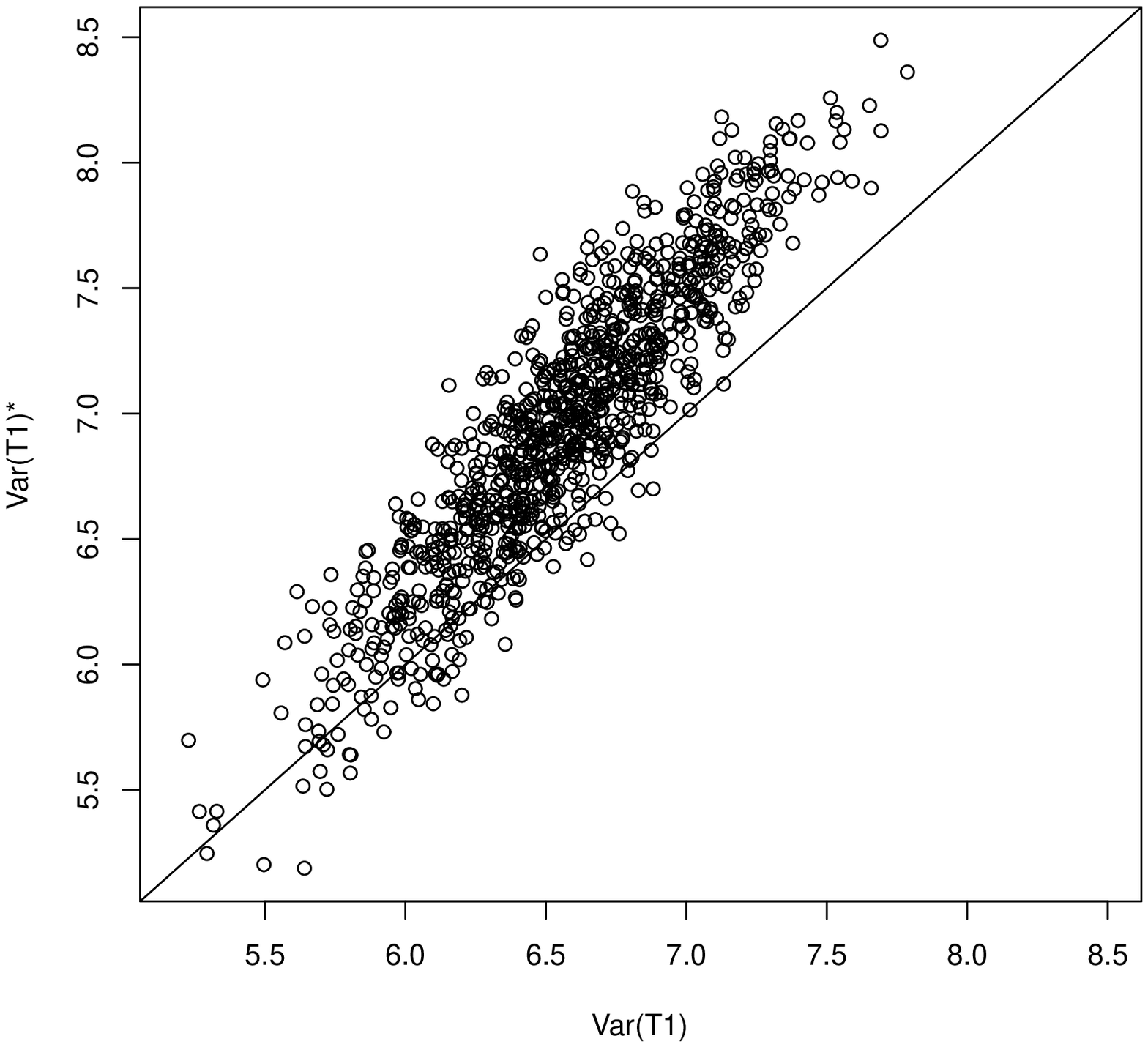}
\label{fig:variance}
\end{center}
\end{figure}

%
%
We present simulation studies comparing the performance of $T_{new1}$, $T_{new2}$ and existing tests.
They have different performances depending on different situations.

\section{Simulations}
In this section, we present simulations studies to compare our proposed tests with existing procedures.
\label{sec:simulation_homogeneity}

We first adopt the following simulation set up and evaluate our proposed tests.
Let us define
\begin{eqnarray*}
\mathb{n}_8 &=& 20(2,2^2,2^3,2^4,2^5,2^6,2^7,2^8) \\
\mathb{n}_{40} &=& 20(\mathb{n}_{1}^*,\mathb{n}_{2}^*,\ldots, \mathb{n}_{8}^*) = 20(2,\ldots,2,2^2,\ldots,2^2,\ldots, 2^8,\ldots,2^8)
\end{eqnarray*}
where  $\mathb{n}_{m}^*=(2^m,2^m,\ldots, 2^m)$ is a $8$ dimensional vector.
We consider the following simulations.
\begin{enumerate}
\item[Setup 1] $\pi_i=0.001$ for $1\leq i \leq k-1$ and $\pi_k=0.001+\delta$ for $k=8$ and $\mathb{n}_8$ 
\item[Setup 2] $\pi_i=0.001+\delta$ for $k=1$ and $\pi_i = 0.001$ for $2\leq i \leq k$ for $k=8$ and $\mathb{n}_8$ 
\item[Setup 3]  $\pi_1 = 0.001 + \delta$ and  $\pi_i=0.001$ for $2\leq i \leq 8$, $k=8$, $n_i= 2560$ for $1\leq i \leq 8$
\item[Setup 4] $\pi_i=0.001$ for $1\leq i \leq k-1$ and $\pi_k=0.001+\delta$ for $k=40$ and $\mathb{n}_{40}$ 
\item[Setup 5] $\pi_i=0.001+\delta$ for $k=1$ and $\pi_i = 0.001$ for $2\leq i \leq k$ for $k=40$ and $\mathb{n}_{40}$ 
\item[Setup 6] $\pi_i=0.001+\delta$ for $i=1$  and $\pi_i=0.001$ for $2 \leq i \leq k$. $n_i= 2560$ for $1\leq i \leq 40$
\end{enumerate}
As test statistics, we use $T_{new1}$,  $M_1$, $T_{new2}$, $M_2$, TS, modTS and PW.
Here,  as discussed in Remark \ref{remark:mleV},
$M_1$ uses $\hat{\cal V}_1^{mle}$  as an estimator of ${\cal V}_1$ in $T_{new1}$ and
$M_2$ uses $\hat{\cal V}_{1*}^{mle}$ for ${\cal V}_{1*}$ in $T_{new2}$.
TS represents the test in (2) and modTS represents the test in (10).
Chi represents  chi-square test based on $T_S > \chi^2_{k-1, 1-\alpha}$
where $\chi^2_{k-1,1-\alpha}$ is the $(1-\alpha)$ quantile of chisquare distribution with degrees of freedom $k-1$.
PW is the test in Potthoff and Whittinghill (1966)
and BL represents the test in Bathke and Lankowski (2005). Note that BL is available only when
sample sizes are all equal.
For calculation of size and power of each test, we simulate 10,000 samples and compute empirical size and power based on 10,000 p values.

{From the above scenario, we consider inhomogeneous sample sizes (Setup 1,2,4 and 5) and
homogeneous sample sizes (Setup 3 and 6). Furthermore, when sample sizes are inhomogeneous,
two cases are considered : one is the case that different $\pi_i$ occurs for a study with large sample (Setup 1 and 4) and
the other for a study with small sample (Setup 2 and 5).
Setup 1-6 consider the cases that only one study has a different probability ($0.001+\delta$)
and all the others have the same probability (0.001).
On the other hand, we may consider the following cases which represent all probabilities are different from each other.
\begin{enumerate}
\item[Setup 7]  $\pi_i=0.001(1+\epsilon_i)$, $k=40$, $n_i = 2560$  for $1\leq i \leq 40$
where $\epsilon_i$s are equally spaced grid in $[-\delta, \delta]$.
\item[Setup 8]  $\pi_i=0.01(1+\epsilon_i)$, $k=40$, $\mathb{n}_{40}^*$
where $\epsilon_i$s are equally spaced grid in $[-\delta, \delta]$.
\end{enumerate}
}

%
%
\begin{table}
\begin{center}
\begin{tabular}{c|cccccccc}
\hline
$\delta$&$T_{new1}$&$M1$ &$T_{new2}$&$M2$ &TS& modTS& Chi& PW\\
\hline
0.000& 0.009& 0.005&  0.029&0.022& 0.114 & 0.048 & 0.103 & 0.006  \\
0.001& 0.070& 0.052&  0.023&0.022& 0.066 & 0.029 & 0.060 & 0.000  \\
0.002& 0.249& 0.184&  0.092&0.091& 0.053 & 0.025 & 0.048 & 0.001 \\
0.003& 0.490& 0.375&  0.253&0.251& 0.057 & 0.022 & 0.046 & 0.022 \\
0.004& 0.688& 0.562&  0.455&0.449& 0.112 & 0.032 & 0.082 & 0.085 \\
0.005& 0.838& 0.717&  0.648&0.642& 0.217 & 0.073 & 0.169 & 0.217 \\
0.006& 0.925& 0.831&  0.803&0.797& 0.391 & 0.170 & 0.315 & 0.397 \\
0.007& 0.966& 0.895&  0.897&0.893& 0.561 & 0.312 & 0.490 & 0.588 \\
0.008& 0.987& 0.936&  0.953&0.950& 0.717 & 0.487 & 0.649 & 0.723 \\
0.009& 0.995& 0.964&  0.979&0.977&  0.835 & 0.651& 0.786 & 0.841 \\
\hline
\end{tabular}
\caption{Powers under Setup 1. The cases of $\delta=0$ represent Type I errors of tests. $M_1$ uses $\hat{\cal V}_1^{mle}$  as an estimator of ${\cal V}_1$ in $T_{new1}$ and
$M_2$ uses $\hat{\cal V}_{1*}^{mle}$ for ${\cal V}_{1*}$ in $T_{new2}$.
TS represents the test in (2) and modTS represents the test in (10).
Chi represents  chi-square test.
PW is the test in Potthoff and Whittinghill (1966) }
\end{center}
\end{table}
 \begin{table}
 \begin{center}
 \begin{tabular}{c|cccccccc}
 \hline
 $\delta$&$T_{new1}$&$M1$ &$T_{new2}$&$M2$ &TS& modTS& Chi& PW \\
 \hline
 0.00 &  0.009 &  0.006 &   0.029 &0.023 &  0.110 & 0.048 & 0.097  & 0.005  \\
 0.01 &  0.009 &  0.005 &   0.043 &0.038 &  0.130 & 0.065 & 0.117  & 0.009  \\
 0.02 &  0.014 &  0.004 &   0.091 &0.087 &  0.149 & 0.085 & 0.138  & 0.010 \\
 0.03 &  0.027 &  0.007 &   0.140 &0.137 &  0.168 & 0.107 & 0.155  & 0.011 \\
 0.04 &  0.054 &  0.011 &   0.213 &0.209 &  0.182 & 0.121 & 0.171  & 0.018 \\
 0.05 &  0.083 &  0.020 &   0.284 &0.282 &  0.191 & 0.136 & 0.181  & 0.027 \\
 0.06 &  0.122 &  0.033 &   0.359 &0.357 &  0.216 & 0.157 & 0.206  & 0.034 \\
 0.07 &  0.168 &  0.053 &   0.432 &0.430 &  0.236 & 0.178 & 0.226  & 0.045 \\
 0.08 &  0.214 &  0.073 &   0.495 &0.494 &  0.249 & 0.195 & 0.238  & 0.069 \\
 0.09 &  0.274 &  0.103 &   0.566 &0.565 &  0.260 & 0.202 & 0.248 & 0.092 \\
 \hline
 \end{tabular}
 \caption{Powers under Setup 2. The cases of $\delta=0$ represent Type I errors of tests.}
 \end{center}
 \end{table}
 \begin{table}[ht]
\centering
\begin{tabular}{rrrrrrrrrr}
  \hline
   $\delta$ &$T_{new1}$&$M1$ &$T_{new2}$ &$M2$&TS& modTS& Chi& PW &BL \\
  \hline
0.0000  & 0.036 & 0.023 & 0.054 & 0.060 & 0.040 & 0.034 & 0.030 & 0.018&0.065   \\
0.0005  & 0.057 & 0.041 & 0.080 & 0.085 & 0.057 & 0.050 & 0.045 & 0.032&0.099   \\
0.0010  & 0.123 & 0.099 & 0.152 & 0.158 & 0.119 & 0.106 & 0.095 & 0.078&0.186   \\
0.0015  & 0.244 & 0.207 & 0.283 & 0.291 & 0.229 & 0.209 & 0.193 & 0.175&0.315   \\
0.0020  & 0.388 & 0.345 & 0.430 & 0.436 & 0.379 & 0.358 & 0.341 & 0.309&0.459   \\
0.0025  & 0.545 & 0.498 & 0.580 & 0.585 & 0.537 & 0.513 & 0.492 & 0.461&0.614   \\
0.0030  & 0.669 & 0.631 & 0.696 & 0.700 & 0.671 & 0.649 & 0.632 & 0.598&0.738   \\
0.0035  & 0.789 & 0.760 & 0.813 & 0.815 & 0.790 & 0.775 & 0.756 & 0.726&0.839   \\
0.0040  & 0.863 & 0.842 & 0.880 & 0.882 & 0.863 & 0.853 & 0.840 & 0.816&0.900   \\
0.0045  & 0.922 & 0.909 & 0.932 & 0.933 & 0.919 & 0.913 & 0.903 & 0.893&0.945   \\
   \hline
\end{tabular}
\caption{Powers under Setup 3. The cases of $\delta=0$ represent Type I errors of tests.
BL represents the test in Bathke and Lankowski (2005).}
\end{table}

%
%
%
%
%
%
%
%
%
%
%

 \begin{table}[ht]
 \centering
 \begin{tabular}{rrrrrrrrr}
   \hline
    $\delta$ &$T_{new1}$&$M1$ &$T_{new2}$ &$M2$&TS& modTS& Chi & PW \\
   \hline
 0.000 & 0.022 & 0.003 & 0.042 & 0.042&  0.196 & 0.060 & 0.186  &0.016  \\
 0.001 & 0.080 & 0.018 & 0.067 & 0.069&  0.160 & 0.048 & 0.151  &0.088  \\
 0.002 & 0.285 & 0.090 & 0.202 & 0.204&  0.199 & 0.057 & 0.189  &0.311  \\
 0.003 & 0.562 & 0.242 & 0.445 & 0.448&  0.296 & 0.099 & 0.282  &0.603  \\
 0.004 & 0.787 & 0.441 & 0.690 & 0.694&  0.462 & 0.185 & 0.442  &0.829  \\
 0.005 & 0.919 & 0.623 & 0.864 & 0.866&  0.659 & 0.355 & 0.640  &0.939  \\
 0.006 & 0.971 & 0.765 & 0.946 & 0.947&  0.804 & 0.542 & 0.791  &0.983  \\
 0.007 & 0.991 & 0.857 & 0.982 & 0.982&  0.913 & 0.723 & 0.906  &0.995  \\
 0.008 & 0.998 & 0.928 & 0.995 & 0.995&  0.964 & 0.855 & 0.960  &0.999  \\
 0.009 & 1.000 & 0.963 & 0.999 & 0.999&  0.989 & 0.934 & 0.987  &0.999  \\
    \hline
 \end{tabular}
  \caption{Powers under Setup 4. The cases of $\delta=0$ represent Type I errors of tests.}
 \end{table}

   \begin{table}[ht]
   \centering
   \begin{tabular}{rrrrrrrrrr}
     \hline
    $\delta$ &$T_{new1}$&$M1$ &$T_{new2}$ &$M2$&TS&modTS&Chi& PW\\
     \hline
    0.000 &0.021 & 0.004 & 0.047 & 0.045    & 0.186 & 0.059 & 0.179  &0.021   \\
    0.002 &0.028 & 0.005 & 0.102 & 0.101    & 0.216 & 0.081 & 0.207  &0.017   \\
    0.004 &0.059 & 0.015 & 0.221 & 0.221    & 0.252 & 0.103 & 0.243  &0.021   \\
    0.006 &0.130 & 0.040 & 0.371 & 0.370    & 0.280 & 0.118 & 0.270  &0.028   \\
    0.008 &0.232 & 0.097 & 0.507 & 0.506    & 0.313 & 0.144 & 0.305  &0.045   \\
    0.010 & 0.335 & 0.158 & 0.626 & 0.626    & 0.339 & 0.156 & 0.331  &0.061   \\
    0.012 & 0.454 & 0.252 & 0.730 & 0.729    & 0.364 & 0.175 & 0.356  &0.091   \\
    0.014 & 0.553 & 0.339 & 0.800 & 0.800    & 0.383 & 0.189 & 0.373  &0.126   \\
      \hline
   \end{tabular}
    \caption{Powers under Setup 5. The cases of $\delta=0$ represent Type I errors of tests.}
   \end{table}

    \begin{table}[ht]
    \centering
    \begin{tabular}{rrrrrrrrrr}
      \hline
       $\delta$ &$T_{new1}$&$M1$ &$T_{new2}$ &$M2$& TS& modTS&Chi & PW & BL\\
      \hline
     0.000 & 0.049 & 0.029 & 0.058 & 0.059 &  0.048 & 0.038 & 0.041 & 0.032& 0.061\\
     0.001 & 0.093 & 0.065 & 0.107 & 0.108 &  0.093 & 0.079 & 0.083 & 0.067&0.114\\
     0.002 & 0.273 & 0.222 & 0.297 & 0.299 &  0.271 & 0.240 & 0.249 & 0.236&0.300\\
     0.003 & 0.535 & 0.479 & 0.560 & 0.562 &  0.535 & 0.504 & 0.512 & 0.511&0.568\\
     0.004 & 0.776 & 0.736 & 0.793 & 0.795 &  0.777 & 0.756 & 0.761 & 0.739&0.803\\
     0.005 & 0.902 & 0.884 & 0.910 & 0.911 &  0.911 & 0.901 & 0.903 & 0.891&0.921\\
     0.006 & 0.966 & 0.957 & 0.969 & 0.969 &  0.966 & 0.961 & 0.963 & 0.966&0.974\\
       \hline
    \end{tabular}
     \caption{Powers under Setup 6. The cases of $\delta=0$ represent Type I errors of tests.}
    \end{table}

%
%
%

\begin{table}[ht]
\centering
\begin{tabular}{rrrrrrrrrrr}
  \hline
 $\delta$ &$T_{new1}$&$M1$ &$T_{new2}$ &$M2$&TS& modTS &Chi & PW &BL\\
  \hline
   0 &  0.044 &  0.027 &  0.053 &  0.053 &   0.046 &  0.036 &  0.039  &  0.031&0.066
    \\
 .25 &  0.080 &  0.052 &  0.096 &  0.094 &   0.084 &  0.069 &  0.072  &  0.061&0.102 \\
 .50 &  0.240 &  0.182 &  0.271 &  0.268 &   0.229 &  0.195 &  0.205  &  0.200&0.280 \\
 .75 &  0.596 &  0.513 &  0.633 &  0.630 &   0.601 &  0.553 &  0.569  &  0.541&0.645 \\
1.00 &  0.927 &  0.889 &  0.941 &  0.940 &   0.930 &  0.911 &  0.917  &  0.904&0.945 \\
\hline
\end{tabular}
\caption{ Powers under Setup 7. The cases of $\delta=0$ represent Type I errors of tests.}
\end{table}

\begin{table}[ht]
\centering
\begin{tabular}{rrrrrrrrrr}
  \hline
 $\delta$ &$T_{new1}$&$M1$ &$T_{new2}$ &$M2$&TS& modTS &Chi & PW\\
  \hline
  0.00& 0.047& 0.025& 0.059& 0.059& 0.073& 0.051 &0.073& 0.030\\
  0.25& 0.123& 0.079& 0.089& 0.089& 0.026& 0.017 &0.026& 0.039\\
  0.50& 0.487& 0.409& 0.353& 0.353& 0.061& 0.044 &0.061& 0.088\\
  0.75& 0.893& 0.858& 0.793& 0.792& 0.265& 0.222 &0.265& 0.179\\
  1.00& 0.996& 0.994& 0.985& 0.985& 0.721& 0.673 &0.721& 0.355\\
\hline
\end{tabular}
\caption{ Powers under Setup 8. The cases of $\delta=0$ represent Type I errors of tests.}
\end{table}

From our simulations, we first see that
$T_{new1}$ obtains more powers than $M_1$ while  $T_{new2}$ and $M_2$ obtain almost similar powers.
The performance of $T_{new1}$ and $T_{new2}$ are different depending on different situations.
when sample sizes are homogeneous (Setup 3, 6 and 7),  $T_{new2}$ obtains slightly more power than $T_{news}$ as shown in $(1)$ in  Theorem \ref{thm:powercomparison}.
On the other hand, when sample sizes are inhomogeneous,  $T_{new1}$ seems to have more advantage for the cases
that different probability occurs for large sample sizes while $T_{new2}$ seems to obtain better powers for the opposite case.
Overall, the performances of $T_{new1}$ and $T_{new2}$ are different depending on situations.
Cochran's test seems to fail in controlling a given size, however the modified TS
achieves reasonable empirical sizes.  When sample sizes are homogeneous, the modified TS has comparable powers, however
for inhomogeneous sample sizes, the modified TS has significantly small powers compare to $T_{new1}$ and $T_{new2}$ for Setup 8.


\begin{table}[ht]
	\centering
	\begin{tabular}{rrrrrrrrrr}
		\hline
 $\delta$ &$T_{new1}$&$M1$ &$T_{new2}$ &$M2$&TS& modTS &Chi & PW & BL  \\
 \hline
   0.0 &0.0507& 0.0118& 0.0563& 0.0563 &0.0544& 0.0234& 0.0544& 0.0499& 0.0548  \\
   0.2 &0.1119& 0.0334& 0.1197& 0.1197 &0.1171& 0.0659& 0.1171& 0.1110& 0.1178  \\
   0.4 &0.5031& 0.2796& 0.5205& 0.5204 &0.5142& 0.3868& 0.5142& 0.5014& 0.5157  \\
   0.6 &0.9709& 0.9012& 0.9730& 0.9730 &0.9727& 0.9425& 0.9727& 0.9706& 0.9728  \\
   0.8 &1.0000& 1.0000& 1.0000& 1.0000 &1.0000& 1.0000& 1.0000& 1.0000& 1.0000  \\
   \hline
 \end{tabular}
 \caption{ Powers under Setup 9. The cases of $\delta=0$ represent Type I errors of tests.}
 \end{table}

\begin{table}[ht]
	\centering
	\begin{tabular}{rrrrrrrrrr}
		\hline
 $\delta$ &$T_{new1}$&$M1$ &$T_{new2}$ &$M2$&TS& modTS &Chi & PW\\
 \hline
   0.0 &0.034 &0.000& 0.055& 0.055& 0.182& 0.046& 0.182& 0.043 \\
   0.2 &0.055 &0.000& 0.050& 0.050& 0.002& 0.000& 0.002& 0.055 \\
   0.4 &0.164 &0.003& 0.101& 0.101& 0.000& 0.000& 0.000& 0.081 \\
   0.6 &0.458 &0.042& 0.278& 0.278& 0.000& 0.000& 0.000& 0.122 \\
   0.8 &0.840 &0.360& 0.650& 0.650& 0.000& 0.000& 0.000& 0.200 \\
   1.0 &0.985 &0.875& 0.933& 0.933& 0.000& 0.000& 0.000& 0.309 \\
	\hline
\end{tabular}
\caption{ Powers under Setup 10. The cases of $\delta=0$ represent Type I errors of tests.}
\end{table}

As suggested by a reviewer, we consider the following two more numerical studies
when $k$ is extremely large.
\begin{enumerate}
\item[Setup 9]  $\pi_i=0.01(1+\epsilon_i)$, $k=2,000$, $n_i = 100$  for $1\leq i \leq 2,000$
where $\epsilon_i$s are equally spaced grid in $[-\delta, \delta]$.
\item[Setup 10]  $\pi_i=0.01(1+\epsilon_i)$, $k=2,000$,
$\mathb{n}=(\mathb{n}_{1,250}, \mathb{n}_{2,250},\ldots, \mathb{n}_{8,250})$
       where  $\mathb{n}_{m,250}=(2^m, 2^m, \ldots, 2^m)$ is a $250$ dimensional vector with all components $2^m$
       and $\epsilon_i$s are equally spaced grid in $[-\delta, \delta]$.
\end{enumerate}
Setup 9 is the case of a extremely large number of groups with small sample sizes.
As mentioned in the introduction, we focus on sparse count data in the sense that $\pi_i$s are small, so
we take  $\pi_i=0.01$ and homogeneous sample sizes $n_i=100$ so that we have $E(X_i)=n_i\pi_i$ which represents very sparse data in each group.
For the number of groups, we use $k=2,000$ which is much larger than $n_i=100$.
Table 9 shows sizes and powers of all tests and we see that
all tests have similar performances when sample sizes are homogeneous.
On the other hand, for the case that sample sizes are highly unbalanced which is the case of Setup 10,
Table 10 shows that
our proposed tests control the nominal level of size and obtain increasing patter of powers while
tests based on chi-square statistics fail in controlling the nominal level of size and obtaining powers.
In particular, those chi-square based tests have decreasing patterns of powers
even though the effect sizes ($\delta$ in this case) increases.
PW controls the size and has increasing pattern of powers, however
the powers of PW are much smaller than those of our proposed tests.
All codes will be available upon request.


%
%
%
%
\section{Real Examples}
In this section, we provide real examples for testing the homogeneity of binomial proportions
from a large number of independent groups.
We apply our proposed tests and existing tests to
the rosiglitazone
data  in Nissen and Wolski (2007).
The data set includes the 42 studies and
consists of
 study size ($N$), number of
myocardial infarctions ($MI$) and number of deaths ($D$) for rosiglitazone (treatment) and
the corresponding results under control arm  for each study.

%

We consider testing (\ref{eqn:hypothesis}) for the proportions of myocardial infarctions and    death rate ($D$)  from cardiovascular causes.
There are four situations,  $(i)$ MI/Rosiglitazone, $(ii)$ Death from Cardiovascular(DCV)/Rosiglitazone,
$(iii)$ MI/Control and $(iV)$ Death from Cardiovascular(DCV)/Control.
Table \ref{tab:homo} shows the $p$-values for different situations and different test statistics.
In case of MI/Rosiglitazone and MI/Control,  all tests have 0 p-value. On the other hand, for the other two cases,
some tests have different results. For DCV/Rosiglitazone, $T_{new2}$, TS and modTS have
small $p$-values while $T_{new1}$ and PW have slightly larger p-values.
For DCV/Control, $T_{new1}$ and $T_{new2}$ have much small p-values (0.107 and 0.079) compared to
$T_S$, modTS, Chi and PW (0.609, 0.406, 0.584  and 0.229, respectively).

\begin{table}[ht]
\begin{center}
\begin{tabular}{c|cccccccc}
\hline
          &$T_{new1}$ &$M1$   &$T_{new2}$&$M2$  &TS        &modTS&Chi   & PW     \\
\hline
MI/Rosig  &0.000      & 0.000 & 0.000   & 0.000 & 0.000    &0.000& 0.000& 0.000  \\
DCV/Rosig &0.063      & 0.133 & 0.003   & 0.004 & 0.000    &0.004& 0.002& 0.059 \\
MI/Cont   &0.000      & 0.000 & 0.000   & 0.000 & 0.000    &0.000& 0.000& 0.000  \\
DCV/Cont  &0.107      & 0.242 & 0.079   & 0.084 & 0.609    &0.406& 0.584& 0.229  \\
\hline
\end{tabular}
\caption {p-values  for homogeneity tests.
Rosig=Rosiglitazone Group, Cont=Control Group,
MI=Myocardial Infarction, DCV=Death from Cardiovascular.}
\end{center}
\label{tab:homo}
\end{table}

\section{Concluding Remarks}
In this paper, we considered testing homogeneity of  binomial proportions  from
a large number of independent studies.  In particular, we focused on the sparse data and heterogeneous sample sizes which may
affect the identification of null distributions.
We proposed new tests
  and showed their asymptotic results under some regular conditions.
We provided simulations and real data examples which show that our proposed tests are convincing in case of sparse and a large number of studies.
  This is a convincing result since our proposed test
  is most reliable in controlling a given size from our simulations,
  so small p-values from our proposed test is strong evidence against the null hypotheses.

\appendix
\section*{Appendix}

\section{Proof of Theorem \ref{thm:modTS}}
We use the Lyapounov's condition for the asymptotic normality of $\frac{{\cal T}_S- E({\cal T}_S)}{\sqrt{{\cal B}_k}}$.
Let ${\cal T}_{Si} = \frac{(X_i - n_i \bar \pi)^2}{n_i \bar \pi (1-\bar \pi)}$, then we define
 ${\cal D}_i  = {\cal T}_{Si} - E({\cal T}_{Si})  =   \frac{(X_i - n_i \bar \pi)^2}{n_i \bar \pi (1-\bar \pi)} - \frac{n_i(\pi_i -\bar \pi)^2}{n_i \bar \pi (1-\bar \pi)}
- \frac{\vari}{n_i \bar \pi (1-\bar \pi)} = \frac{1}{n_i \bar \pi (1-\bar \pi) } ((X_i - n_i \pi_i)^2 + 2n_i (X_i - n_i \pi_i) (\pi_i -\bar \pi) - n_i \vari )$.  We show that the Lyapounov's condition is satisfied,
$ \frac{\sumk E({\cal D}_i^4) }{{\cal B}_k^2} \rightarrow 0.$
We see that
\begin{eqnarray*}
 \frac{\sumk E({\cal D}_i^4)}{{\cal B}_k^2} &\leq&
\frac{1}{{\cal B}_k^2}\sumk \frac{  n_i^4 E(\hat \pi_i - \pi_i)^8  + 2^4 n_i^4 (\pi_i -\bar \pi)^4  n_i^4  E(\hat \pi_i - \pi_i)^4 + n_i^4 \vari^4   }{n_i^4(\bar \pi (1-\bar \pi))^4}  \\
&=& \frac{1}{ (\bar \pi (1-\bar \pi))^4 {\cal B}_k^2 } \sumk  \left[\left(\theta_i^4 + \frac{\theta_i}{n_i} \right)
+ n_i^2 (\pi_i -\bar \pi)^4 (3\theta_i^2 + \frac{(1-6\theta_i)\theta_i}{n_i}) + \theta_i^4\right]    \\
&\leq & \frac{\sumk \left(2\theta_i^4 + \frac{\theta_i}{n_i} \right)}{(\bar \pi (1-\bar \pi))^4 {\cal B}_k^2 }
+ \frac{ 3\sumk n_i^2 \theta_i (\pi_i -\bar \pi)^4 (\theta_i + \frac{1}{n_i}) }{(\bar \pi (1-\bar \pi))^4{\cal B}_k^2} \\
&\rightarrow& 0
\end{eqnarray*}
from the given conditions. Therefore, we have the asymptotic normality of $\frac{{\cal T}_S-E({\cal T}_S)}{\sqrt{{\cal B}_k}} \rightarrow N(0,1) $ in distribution.
Furthermore, we also have the asymptotic normality of
\begin{eqnarray*}
T_0 = \frac{T_S-k}{\sqrt{{\cal B}_{0k}}} = \sqrt{\frac{{\cal B}_k}{{\cal B}_{0k}}} \frac{T_S-k}{\sqrt{{\cal B}_k}}
+ \frac{k-E(T_S)}{\sqrt{{\cal B}_{0k}}}
= \sigma_k \frac{T_S-k}{\sqrt{{\cal B}_{k}}} + \mu_k
\end{eqnarray*}
which leads to
$ P( T_0 \geq z_{1-\alpha})  = P(\sigma_k \frac{T_S-k}{\sqrt{{\cal B}_{k}}} + \mu_k  \geq z_{1-\alpha}    ) = P(\frac{T_S -k}{\sqrt{{\cal B}_k}} \geq \frac{z_{1-\alpha}}{\sigma_k} -\mu_k)$. Using  $\frac{T_S -k}{\sqrt{{\cal B}_k}}  \rightarrow N(0,1)$ in distribution,    we have   $ P(T_0 \geq z_{1-\alpha}) - \bar \Phi( \frac{z_{1-\alpha}}{\sigma_k} -\mu_k) \rightarrow 0.$
   \qed

%

\section{Proof of Lemma \ref{lemma:varianceT1}}
Since   $T_{1i}$ and $T_{1j}$ for $i\neq j$ are independent,  we have ${\cal V}_1 \equiv Var(T_1) = \sumk Var(T_{1i})$ where
\begin{eqnarray*}
Var(T_{1i}) &=& n_i^2 Var[(\hat \pi_i - \pi_i)^2] + d_i^2 Var[\hatvari] + 4 n_i^2 (\pi_i-\bar \pi)^2  Var[ (\hat \pi_i-\pi_i)] \\
 &&-2n_i d_i Cov((\hat \pi_i - \pi_i)^2, \hat \pi_i (1-\hat \pi_i)) \\
  && + 2 Cov(n_i(\hat \pi_i-\pi_i)^2, 2n_i (\hat \pi_i -\pi_i)(\pi_i-\bar \pi))-2 Cov( 2n_i(\hat \pi_i - \pi_i)(\pi_i-\bar \pi), d_i\hatvari).
\end{eqnarray*}
Using the following results
\begin{eqnarray*}
Var[(\hat \pi_i - \pi_i)^2]  &=&  E[(\hat \pi_i - \pi_i)^4] - (E[(\hat \pi_i - \pi_i)^2])^2 \\
&=&  \frac{2\theta_i^2}{n_i^2} + \frac{(1-6\theta_i)\theta_i}{n_i^3} \\
Var[\hatvari] &=&   \frac{(1-\theta_i)\theta_i}{n_i} - \frac{2\theta_i(1-4\theta_i)}{n_i^2} + \frac{(1-6\theta_i)\theta_i}{n_i^3} \\
Cov( (\hat \pi_i - \pi_i)^2, \hatvari) &=& \frac{n_i-1}{n_i^3} \theta_i(1-6\theta_i)\\
Cov((\hat \pi_i- \pi_i)^2, \hat \pi_i-\pi_i) &=& E(\hat \pi_i-\pi_i)^3 = \frac{(1-2\pi_i)\theta_i}{n_i^2} \\
Cov( (\hat \pi_i -\pi_i), \hatvari) &=& \frac{(1-2\pi_i) \theta_i}{n_i}\left(1 -\frac{1}{n_i} \right),
\end{eqnarray*}
 we derive
\begin{eqnarray*}
&&Var(T_1)  \\
&=&  \sum_{i=1}^k  \left\{ \theta_i^2 \left(2-\frac{6}{n_i} - \frac{d_i^2}{n_i} + \frac{8d_i^2}{n_i^2}-\frac{6d_i^2}{n_i^3} + 12 d_i \frac{n_i-1}{n_i^2} \right) +  \theta_i \left(\frac{1}{n_i} +\frac{d_i^2}{n_i} - \frac{2d_i^2}{n_i^2}  + \frac{d_i^2}{n_i^3}- 2d_i\frac{n_i-1}{n_i^2} \right)   \right\}\\
&&+ 4 \sum_{i=1}^k n_i(\pi_i - \bar \pi)^2 \theta_i + \frac{4}{N}\sumk n_i(\pi_i - \bar \pi)(1-2\pi_i)\theta_i  \\
&=&  \sumk {\cal A}_{1i}  \theta_i^2 + \sumk {\cal A}_{2i} \theta_i + 4 \sum_{i=1}^k n_i(\pi_i - \bar \pi)^2 \theta_i
+\frac{4}{N}\sumk n_i(\pi_i - \bar \pi)(1-2\pi_i)\theta_i
\end{eqnarray*}
where $ {\cal A}_{1i}= \left( 2-\frac{6}{n_i} - \frac{d_i^2}{n_i} + \frac{8d_i^2}{n_i^2}-\frac{6d_i^2}{n_i^3} + 12 d_i \frac{n_i-1}{n_i^2}  \right)$
and ${\cal A}_{2i} =\left(\frac{1}{n_i} +\frac{d_i^2}{n_i} - \frac{2d_i^2}{n_i^2}  + \frac{d_i^2}{n_i^3}- 2d_i\frac{n_i-1}{n_i^2} \right) = \frac{n_i}{N^2}$ from $d_i = \frac{n_i}{n_i-1} \left(1 -\frac{n_i}{N} \right)$.

\section{Proof of Lemma \ref{lemma:vart1} }
\begin{enumerate}
\item
Using $d_i = \frac{n_i}{n_i-1} (1-\frac{n_i}{N}) < 2$, we can derive ${\cal A}_{1i}$ is uniformly bounded since
${\cal A}_{1i} = 2   - \frac{6}{n_i} - \frac{6 d_i^2}{n_i} + \frac{8d_i^2}{n_i^2} -\frac{6d_i^2}{n_i^3}  + 12 \frac{n_i}{n_i-1} \frac{n_i-1}{n_i^2}(1-\frac{n_i}{N}) = 2 + \frac{6}{n_i} -\frac{12}{N}
+ \frac{d_i^2}{n_i}( -1 + \frac{8}{n_i} -\frac{6}{n_i^2}) (1-\frac{n_i}{N}) $.
Let $x = \frac{1}{n_i} \leq \frac{1}{2}$, then
$ f(x) = (-1 + \frac{8}{n_i} -\frac{6}{n_i^2}) =   -6(x-\frac{2}{3})^2 + \frac{7}{9}$ which has the value
$ -1< f(x) \leq \frac{3}{2}$. Therefore,  we have
$  2 + \frac{6}{n_i} -\frac{12}{N} + \frac{6}{n_i} \geq   {\cal A}_{1i}  \geq 2 + \frac{6}{n_i} -\frac{12}{N} -  \frac{4}{n_i}$.
Using  $n_i\geq 2$ and $N \rightarrow \infty$ as $k \rightarrow \infty$,
  lower and upper bound are uniformly bounded away from 0 and $\infty$ for all $i$.  Therefore,  we have ${\cal A}_{1i} \asymp  1$
  and ${\cal A}_{2i}   =   \frac{n_i}{N^2}$ leading to
      ${\cal V}_1 =  \sumk {\cal A}_{1i} \theta_i^2 + \sumk {\cal A}_{2i} \theta_i \asymp  \sumk \theta_i^2 +   \frac{1}{N^2}  \sumk n_i \theta_i$.
%
%

\item
Let  ${\cal G}_n   =  4 \sumk  n_i (\pi_i -\bar \pi)^2 \theta_i  +  4 \frac{1}{N} \sumk n_i (\pi_i-\bar \pi)(1-2\pi_i)\theta_i
=  4 \sumk \theta_i G_{i} $ where $G_i =   n_i(\pi_i -\bar \pi)^2 + \frac{n_i}{N}(\pi_i-\bar \pi)(1-2\pi_i)$.
If we define ${\cal B}  = \{ i :  |\pi_i - \bar \pi| \geq \frac{(1+\epsilon)}{N} \}$
for some $\epsilon>0$, then we decompose
\begin{eqnarray}
{\cal V}_1 &=& \underbrace{ \sum_{i \in {\cal B}} ( {\cal A}_{1i} \theta_i^2 + {\cal A}_{2i} \theta_i)}_{{\cal F}_1} +
   \underbrace{\sum_{i \in {\cal B}^c} ( {\cal A}_{1i} \theta_i^2 + {\cal A}_{2i} \theta_i)}_{{\cal F}_2} \\
{\cal G}_n &=&  4 \underbrace{\sum_{i \in {\cal B}} \theta_i G_i}_{{\cal G}_{n1}} +
4 \underbrace{\sum_{i \in {\cal B}^c}\theta_i G_i}_{{\cal G}_{n2}} \equiv
4 {\cal G}_{n1} + 4 {\cal G}_{n2}.
\end{eqnarray}

For $i \in {\cal B}$, we have $ \frac{n_i}{N}| (1-2\pi_i) (\pi_i  - \bar \pi) \theta_i| \leq \frac{n_i}{(1+\epsilon)} (\pi_i - \bar \pi)^2 \theta_i$ which implies
\begin{eqnarray*}
\frac{4\epsilon}{1+\epsilon} \sum_{i \in {\cal B}}  n_i  ( \pi_i -\bar \pi  )^2 \theta_i \leq   4 {\cal G}_{n1} \leq
\frac{4 (2+\epsilon)}{1+\epsilon} \sum_{i \in {\cal B}}  n_i  ( \pi_i -\bar \pi  )^2 \theta_i.
\end{eqnarray*}
This leads to $4 {\cal G}_{n1} \asymp  \sum_{i \in {\cal B}} n_i (\pi_i- \bar \pi)^2 \theta_i$ and
\begin{eqnarray}
       {\cal F}_1+ 4 {\cal G}_{n1} \asymp
        {\cal F}_1+ \sum_{i \in {\cal B}}  n_i  ( \pi_i -\bar \pi  )^2 \theta_i.
     \label{eqn:Gn1}
\end{eqnarray}
For ${\cal B}^c = \{ i |  |\pi_i -\bar \pi| < \frac{(1+\epsilon)} {N}\}$, we first show
$ {\cal F}_2 +  4{\cal G}_{n2}  \geq    \sum_{i \in {\cal B}^c} {\cal A}_{1i} \theta_i^2$.
For $i \in {\cal B}^c$ and $x = \pi_i-\bar \pi$, we have
$ G_i = n_i( x+  \frac{1}{2N} (1-2\pi_i))^2 - \frac{(1-2\pi_i)^2 n_i }{4N^2} \geq  -\frac{ (1-2\pi_i)^2 n_i}{4N^2}$ leading to
\begin{eqnarray}
{\cal F}_2 + 4 {\cal G}_{n2}  &\geq&
  \sum_{i \in {\cal B}^c} {\cal A}_{1i}\theta_i^2  +    \frac{1}{N^2}   \sum_{i \in {\cal B}^c} {n_i\theta_i} \left( 1-(1-2\pi)^2   \right) \nonumber \\
   &=&  \sum_{i\in {\cal B}^c}  {\cal A}_{1i} \theta_i^2 +    \frac{4}{N^2}  \sum_{i \in {\cal B}^c} {n_i \theta_i^2}
   = \sum_{i\in {\cal B}^c}  {\cal A}_{1i} \theta_i^2 + 4 \sum_{i \in {\cal B}^c} {\cal A}_{2i}\theta_i^2 \nonumber \\
   &>& \sum_{i\in {\cal B}^c}  {\cal A}_{1i} \theta_i^2.  \label{eqn:Gn2_1}
\end{eqnarray}
The upper bound of $4{\cal G}_{n_2}$ is
\begin{eqnarray*}
4 {\cal G}_{n_2}     &\leq&  \frac{4(1+\epsilon)}{N^2} \sum_{i \in {\cal B}^c}  n_i \theta_i = 4(1+\epsilon) \sum_k  {\cal A}_{2i} \theta_i
     \label{eqn:Gn2_2}
\end{eqnarray*}
resulting in
\begin{eqnarray}
{\cal F}_2 + 4 {\cal G}_{n2} &\leq&   4(1+\epsilon) \sum_{i\in {\cal B}^c} {\cal A}_{1i} \theta_i^2 +
  4(1+\epsilon) \sum_{i \in {\cal B}^c} {\cal A}_{2i} \theta_i  + 4 \sum_{i \in {\cal B}^c} n_i (\pi_i - \bar \pi)^2 \theta_i   \nonumber \\
  &<& 4(1+\epsilon) ( {\cal F}_2 +  \sum_{i \in {\cal B}^c} n_i (\pi_i - \bar \pi)^2 \theta_i).
 \label{eqn:F2}
\end{eqnarray}
Combining   (\ref{eqn:Gn2_1}) and (\ref{eqn:F2}), we have
\begin{eqnarray}
 \sum_{i \in {\cal B}^c} {\cal A}_{1i} \theta_i^2   <    {\cal F}_2 + {\cal G}_{n2} < 4(1+\epsilon)( {\cal F}_2 +  \sum_{i \in {\cal B}^c} n_i (\pi_i - \bar \pi)^2 \theta_i).
   \label{eqn:Gn2}
\end{eqnarray}

From (\ref{eqn:Gn1}) and (\ref{eqn:Gn2}), we conclude, for $K=4(1+\epsilon)$,
\begin{eqnarray*}
 \sum_{i=1}^k  {\cal A}_{1i} \theta_i^2 < Var(T_1) \leq K (\nu_1 + ||{\mathb \pi} - \bar {\mathb \pi} ||_{{\bf n} \theta}^2).
\end{eqnarray*}
In particular,  if $ {\cal B}^c$ is an empty set,
then  we have $Var(T) =  {\cal F}_{1} + 4 {\cal G}_{n1}$, therefore
(\ref{eqn:Gn1}) implies   (\ref{eqn:asympVar}).
\end{enumerate}

\section{Proof of Lemma \ref{lemma:lemma4}}
Let $X = \sum_{i=1}^n X_i$ where $X_i$s are iid Bernoulli($\pi$).
In expansion of $(X-n\pi)$,
each term has the form of $(X_{i_1}-\pi)^{m_1}(X_{i_2}-\pi)^{m_2}\cdots (X_{i_k}-pi)^{m_k}$
for  $1\leq i_1,\ldots, i_k \leq n$  and $ m_1 + \cdots + m_k=n$, so
if there exists at least one $m_{k}=1$, then expectation of the term is zero.
We only need to consider the terms without $(X_{i_j}-\pi)$, so we finally have
\begin{eqnarray*}
E(X-n\pi)^8 &= &  E(\sum_{i=1}^n (X_i -\pi))^8 \\
&=&  {n \choose 1} E(X_1-\pi)^8\\
&& +  2 {8 \choose 6,2} {n \choose 2} E(X_1-\pi)^6 E(X_1-\pi)^2 \\
&&+  2 {n \choose 2} {8 \choose 5,3} E(X_1-\pi)^5 E(X_1-\pi)^3  \\
&&+  {n\choose 2} {8 \choose 4,4}   [E(X_1-\pi)^4]^2 \\
&&+  \frac{3!}{2!}  {n \choose 3} {8 \choose 4,2,2} E(X_1-\pi)^4 [E(X_1-\pi)^2]^2 \\
&&+ \frac{3!}{2!}  {n \choose 3} {8 \choose 3,3,2} [E(X_1-\pi)^3]^2 E(X_1-\pi)^2 \\
&&+ {n \choose 4} {8 \choose 2,2,2,2} [E(X_1-\pi)^2]^4.
\end{eqnarray*}
We have $E(X_1-\pi)^m  =  \sum_{i=0}^m {m \choose i} E(X_1^i) (-\pi)^{m-i} = (-\pi)^m + \sum_{i=1}^m {m \choose i} E(X_1^i) (-\pi)^{m-i}$
and  using   $E(X_1^i) = E(X_i) =\pi$ for $i\geq 1$, we obtain $E(X_1-\pi)^m = (-\pi)^m + \pi \sum_{i=1}^m {m \choose i} (-\pi)^{m-i}
 = (-\pi)^m-\pi(-\pi)^m +  \pi\sum_{i=0}^m {m \choose i} (-\pi)^m = (1-\pi)(-\pi)^m + \pi(1-\pi)^m = \pi(1-\pi)( (-1)^m \pi^{m-1}+(1-\pi)^{m-1})\leq
 \pi(1-\pi)$ for $m \geq 2$.
Since all coefficients in the expansion of $E(\sum_{i=1}^n (X_i - \pi))$
are fixed constants, for some universal constant $C>0$, we have
\begin{eqnarray*}
E(X-n\pi)^8 &\leq& C \max(n\pi(1-\pi), (n\pi(1-\pi))^2,  (n\pi(1-\pi))^3,  (n\pi(1-\pi))^4 )       \\
&=&  C \max \{ n\pi(1-\pi), (n\pi(1-\pi))^4 \}.
\end{eqnarray*}
since maximum is obtained at either $n\pi(1-\pi)$ or $(n\pi(1-\pi))^4$ depending on $n\pi(1-\pi) \leq 1 $ or $n\pi(1-\pi)>1$.

For the second equation,  we first consider  the moment of $E(\hat \pi^4)$ and $E(\hat (1-\hat \pi)^4)$.
The latter one is easily obtained from the first one by changing the distribution from $B(n,\pi)$ to $B(n,1-\pi)$.
We first obtain
\begin{eqnarray*}
E \hat \pi ^4 &=&  \pi^4 + \frac{6\pi^2 \theta}{n} + \frac{4\pi(1-2\pi)\theta}{n^2} + \frac{3\theta^2}{n^2} + \frac{(1-6\theta)\theta}{n^3} \\
&\leq& \pi^4 + \frac{6\pi^3}{n} + \frac{7\pi^2}{n^2} + \frac{\pi}{n^3}  \\
&\leq&   7 \left( \pi^4 + \frac{\pi^3}{n} + \frac{\pi^2}{ n^2} + \frac{\pi}{n^3} \right) \\
 &=&  28 \max\left( \pi^4, \frac{\pi}{n^3}\right)
\end{eqnarray*}
where the last equality holds due to the fact that
the maximum is obtained at either $\pi^4$ or $\frac{\pi}{n^3}$ depending on $\pi \geq \frac{1}{n}$ or $\pi< \frac{1}{n}$.
Similarly, the following inequality is obtained
\begin{eqnarray*}
E (1-\hat \pi)^4 &\leq& 28 \max \left( (1-\pi)^4, \frac{1-\pi}{n^3} \right).
\end{eqnarray*}
Using $E \hat \pi^4 (1-\hat \pi)^4 \leq  \min(E\hat \pi^4, E (1-\hat\pi)^4 )$, we have
\begin{eqnarray*}
E \hat \pi^4 (1-\hat \pi)^4 &\leq&  \min(E\hat \pi^4, E (1-\hat\pi)^4 ) \\
 &\leq& 28 \min \left\{ \max\left( \pi^4, \frac{\pi}{n^3}\right) ,
  \max \left( (1-\pi)^4, \frac{1-\pi}{n^3} \right)   \right\} \\
 &=&   \max\left( \pi^4, \frac{\pi}{n^3}\right)~~~\mbox{if  $\pi \leq \frac{1}{2}$} \\
 &&    \max \left( (1-\pi)^4, \frac{1-\pi}{n^3} \right)~~~\mbox{if  $\pi > \frac{1}{2}$}.
\end{eqnarray*}
If $\pi \leq \frac{1}{2}$, $\pi \geq  2\pi (1-\pi) = 2\theta$; if $\pi > \frac{1}{2}$, $1-\pi \leq 2\theta$. So the last equality is
\begin{eqnarray*}
E \hat \pi^4 (1-\hat \pi)^4 \leq  C' \max \left( \theta^4, \frac{\theta}{n^3} \right)
\end{eqnarray*}
for some universal constant $C'$.

We use the following relationship: for some constants $b_m$, $m=1,\ldots,l-1$
\begin{eqnarray*}
X^l = \prod_{j=1}^{l} (X-j+1) + \sum_{m=1}^{l-1} b_m  \prod_{j=1}^{m} (X-j+1).
\end{eqnarray*}
For example we have $x^3 = x(x-1)(x-2) + 3x(x-1) +x$.
Using $E\prod_{j=1}^{l} (X-j+1) =  \prod_{j=1}^{l} (n-j+1) \pi^l$,
\begin{eqnarray*}
E\hat \pi^l   &=&  \frac{1}{n^l} E\prod_{j=1}^{l} (X-j+1)  +  \frac{1}{n^l} \sum_{m=1}^{l-1} b_m E \left( \prod_{j=1}^{m} (X-j+1) \right)   \\
&=& \pi^l + O(\frac{\pi^l}{n}) +  O \left( \sum_{m=1}^{l-1}  \frac{\pi^m}{n^{l-m}} \right) \\
&=&  \pi^l + O\left(\frac{\pi^{l-1}}{n} + \frac{\pi}{n^{l-1}}   \right).
\end{eqnarray*}
Using this, we can derive
\begin{eqnarray*}
E(\hat \pi^l -\pi^l)^2 = E \hat \pi^{2l} - 2\pi^l E \hat \pi^l + \pi^{2l} =  O\left(\frac{\pi^{2l-1}}{n} + \frac{\pi}{n^{2l-1}} \right).
\end{eqnarray*}
\begin{eqnarray*}
E(\hat \eta_l - \pi^l)^2 = E(\hat \eta_l -\hat \pi_l +\hat \pi_l - \pi_l)^2
\leq 2^2 E(\hat \eta_l -\hat \pi^l)^2 + 2^2 E(\hat \pi^l - \pi^l)^2.
\end{eqnarray*}

Since $\hat \eta_l -\hat \pi^l = {\hat \pi^l}O(\frac{1}{n})  + \sum_{i=1}^{l-1} {\hat \pi^{l-i}} O(\frac{1}{n^i})$,
 we have $E(\hat \eta_l -\hat \pi^l)^2 \leq  \left\{E(\hat \pi^{2l})O(\frac{1}{n^2}) + \sum_{i=1}^{l-1}  {\hat \pi^{2l-2i}} O(\frac{1}{n^{2i}}) \right\}.$
Using $E \hat \pi ^{2l}=\pi^{2l}+ O \left( \frac{\pi}{n^{2l-1}} + \frac{\pi^{2l-1}}{n} \right)$ from (\ref{eqn:hatpipowerk}),
we obtain
\begin{eqnarray}
E(\hat \eta_l -\hat \pi^l)^2   &=& O\left(\frac{1}{n^2} \right)
\left(\pi^{2l} + O \left( \frac{\pi}{n^{2l-1}} + \frac{\pi^{2l-1}}{n} \right) \right)  \nonumber \\
&&+  \sum_{i=1}^{l-1} \left(\pi^{2l-2i} + O\left( \frac{\pi}{n^{2l-2i-1}}  + \frac{\pi^{2l-2i-1}}{n}  \right) \right) O\left(\frac{1}{n^{2i}} \right)   \nonumber \\
&=& O\left( \sum_{i=1}^{l-1} \frac{\pi^{2(l-i)}}{n^{2i}} + \frac{\pi}{n^{2l-1}}  \right)  \label{eqn:lasthateta}.
\end{eqnarray}
We can show  $\frac{\pi^{2(l-i)}}{n^{i}} \leq  \frac{\pi}{n^{2l-1}} + \frac{\pi^{2l-1}}{n}$
for $2\leq i \leq l-1$ since
$\frac{\pi^{2(l-i)}}{n^{i}} \leq \frac{\pi^{2l-1}}{n}$ for $\pi \geq \frac{1}{n}$ and
$\frac{\pi^{2(l-i)}}{n^{i}} \leq \frac{\pi}{n^{2l-1}}$ for $\pi <\frac{1}{n}$.
Using this, we have $(\ref{eqn:lasthateta})  \leq  O\left( \frac{\pi}{n^{2l-1}} + \frac{\pi^{2l-1}}{n} \right)$
which proves   $E(\hat \eta_l -\hat \pi^l)^2   = O\left( \frac{\pi}{n^{2l-1}} + \frac{\pi^{2l-1}}{n} \right)$.

\section{Proof of Lemma \ref{lemma:ratio}}
 For the ratio consistency of $\hat{\cal{V}}_1$,
it is enough to show $ \frac{E[( \hat {\cal V}_1  - {\cal V}_1 )^2]}{({\cal V}_1)^2} \rightarrow 0$ as $k \rightarrow \infty$.
Since $\hat {\cal V}_1$ is an unbiased estimator of ${\cal V}_1$,
\begin{eqnarray*}
Var(\hat{\cal V}_1) &=& E[(\hat {\cal V}_1 -  {\cal V}_1)^2]\\
 &=& \sumk \sum_{l=1}^4
a^2_{li} E [( \hat \eta_{li}- \eta_{li})^2] + \sum_{i \neq i'} \sum_{l \neq l'} a_{li} a_{l'i'} E[ (\hat \eta_{li} - \eta_{li})(\hat
\eta_{l'i'}-\eta_{l'i'})] \\ &=& \sumk \sum_{l=1}^4 a^2_{li} E [( \hat \eta_{li}- \eta_{li})^2]
\end{eqnarray*}
where the last equality follows since $E[(\hat \eta_{li}- \eta_{li})(\hat \eta_{l'i'} - \eta_{l'i'})] = E[(\hat \eta_{li}-\eta_{li})] E[(\hat
\eta_{l'i'}-\eta_{l'i'})]=0$ because $\hat \eta_{li}$ and $\hat \eta_{l'i'}$ are independent for $i \neq i'$ and both are unbiased estimators.
Since ${\cal V}_1$ depends on
$\theta_i=\pi_i(1-\pi_i)$, we have the same result if we change $\pi_i$ to $1-\pi_i$; in other words,
$Var(\hat {\cal V}_1) =\sum_{i=1}^k \sum_{l=1}^4 a_{li}^2 (\eta_{li} -\eta_{li})^2 =
\sum_{i=1}^k \sum_{l=1}^4 a^2_{li} (\hat \eta^*_{li} -\eta^*_{li})^2$ where
$\eta^*_{li} = (1-\pi_i)^l$ and $\hat \eta^*_{li}$ is the corresponding unbiased estimator.
For $\pi \leq 1/2$, we use ${\cal V}_1 =\sum_{i=1}^k \sum_{l=1}^4 a_{li} \pi_i^l$ and
obtain $Var(\hat {\cal V}_1) =  O(\sum_{i=1}^k  (\frac{\pi_i^3}{n_i} + \frac{\pi_i}{n_i^3}))$ from Lemma \ref{lemma:lemma4}.
Since $ \pi_i \leq \delta <1$, we have $Var(\hat {\cal V}_1) =  O(\sum_{i=1}^k  (\frac{\pi_i^3}{n_i} + \frac{\pi_i}{n_i^3}))
=  O(\sum_{i=1}^k  (\frac{\theta_i^3}{n_i} + \frac{\theta_i}{n_i^3})).$
%
From Lemma \ref{lemma:vart1} and the given condition, we obtain
\begin{eqnarray*}
\frac{Var(\hat{\cal V}_1)}{  {\cal V}_1^2 } &= & O\left(\frac{\sum_{i=1}^k \left(\frac{\theta_i^3}{n_i}+ \frac{\theta_i}{n_i^3}\right)}
{(\sum_{i=1}^k (\theta_i^2  + \frac{1}{N^2} \frac{\theta_i}{n_i} ) )^2} \right) =o(1).
\end{eqnarray*}

Similarly,  we can show, for some constant $C'$,
\begin{eqnarray*}
\frac{Var(\hat{\cal V}_{1*})}{ ({\cal V}_{1*})^2 } =
O\left(\frac{ (\tilde \theta)^3 \sum_{i=1}^k \frac{1}{n_i} + \tilde \theta \sumk \frac{1}{n_i^3} }{  (k (\tilde \theta)^2  + \frac{\tilde \theta }{N^2} \sum_{i=1}^k \frac{1}{n_i})^2 } \right)  =o(1).
\end{eqnarray*}

\section{Proof of Theorem \ref{cor:size}}
Since the condition in Lemma \ref{lemma:ratio} holds,
$\hat {\cal V}_1$  and $\hat {\cal V}_{1*}$ are
the ratio consistent estimator of
$ {\cal V}_1 =  {\cal V}_{1*}$ under the $H_0$.
From $ \frac{T}{\sqrt{{\cal V}_1}} = \frac{T_1 - T_2}{\sqrt{{\cal V}_1}}$, we only need to show
$(i)$
$\frac{T_1}{\sqrt{{\cal V}_1}} \rightarrow N(0,1)$ in distribution and
$(ii)$ $\frac{T_2}{\sqrt{{\cal V}_1}} \rightarrow 0$ in probability.
To prove $(i)$, we show the Lyapounov's condition (see Billingsley(1995))  for the asymptotic normality is satisfied.
In other words, under $H_0$, we need to show  $\frac{\sum_{i=1}^k E(T_{1i}^4)}{ Var(T_1)^2} \rightarrow 0$.
Under $H_0$,  we have $T_{1i} = n_i(\hat \pi_i -\pi_i)^2 - d_i \hatvari$ with $E(T_{1i})=0$, therefore
the Lyapounov's condition is $ \sumk E(T_{1i}^4)/ Var(T_1)^2 \rightarrow 0$.
Using Lemma 4,  we have $\sumk E(T_{1i}^4) \leq 2^4(\sumk n_i^4 E(\hat \pi_i -\pi_i )^4 + d_i^4 E(\hatvari)^4 )
=  O( \sumk (\theta_i^4 + \frac{\theta_i}{n_i})) + O( \sumk (\theta_i^4  + \frac{\theta_i}{n_i}^3) ) =
 O( (k \theta^4 + \frac{\theta} \sumk \frac{1}{n_i}))$ since all $\theta_i =\theta$ under $H_0$.
Combining this with the result 1 in Lemma  \ref{lemma:vart1}, we have
$\frac{\sum_{i=1}^k E(T_{1i}^4)}{ Var(T_1)^2}  =  \frac{ O(k\theta^4 + \theta \sumk \frac{1} {n_i}) }{ (k\theta^2 + \frac{\theta}{N^2} \sumk \frac{1}{n_i})^2} \leq   \frac{ k \theta^4 + \theta\sumk \frac{1}{n_i}}{ k \theta^4} = \frac{1}{k} + \frac{\sumk \frac{1}{n_i}}{ k \theta^3} \rightarrow 0$ as $k \rightarrow \infty$ from the given condition $\frac{\sumk \frac{1}{n_i}}{ k \theta^3} \rightarrow 0$ which shows $\frac{T_1}{\sqrt{{\cal V}_1}} \rightarrow N(0,1)$ in distribution.

Furthermore, from the Lemma \ref{lemma:vart1} under the $H_0$, we have
${\cal V}_1 \asymp   k\theta^2 + \theta \sumk \frac{1}{n_i} $, therefore we obtain
$E\left(\frac{T_2}{\sqrt{{\cal V}_1}} \right) = \frac{E( N(\hat{\bar \pi} - \bar \pi)^2   )}{\sqrt{{\cal V}_1}} \asymp
\frac{ \theta }{\sqrt{k \theta^2  + \frac{\theta}{N^2} \sumk \frac{1}{n_i} }} \leq \frac{1}{\sqrt{k}} \rightarrow 0$
which leads to  $\frac{T_2}{\sqrt{{\cal V}_1}} \rightarrow 0$ in probability.
Combining the asymptotic normality of $\frac{T}{\sqrt{{\cal V}_1}}$
with the ratio consistency of $\hat {\cal V}_1$ and $\hat {\cal V}_{1*}$,
we have the asymptotic normality of $T_{new1}$ and $T_{new2}$ under the $H_0$.

\section{Proof of Theorem \ref{thm:normality}}
 Since $T=T_1 -T_2$ from (\ref{eqn:T1T2}),
we only need to show the following:
\begin{enumerate}
\item [(I)] $\frac{T_1 -\sum_{i=1}^k n_i (\pi_i - \bar \pi)^2 }{\sqrt{{Var(T_1)}}} \rightarrow N(0,1)$ in distribution
\item [(II)]  $\frac{T_2}{\sqrt{{Var(T_1)}}} \rightarrow 0$ in probability.
\end{enumerate}
For $(I)$,  we use the Lyapounov's condition for the asymptotic normality of $T_1$.
We show  $\frac{\sum_{i=1}^k E(T_{1i}- n_i(\pi_i -\bar \pi)^2)^4}{Var(T_1)^2 }  \rightarrow 0$
where  $ G_i = T_{1i}- n_i(\pi_i -\bar \pi)^2 = n_i (\hat \pi_i -\pi_i)^2 -d_i \hat \pi_i (1-\hat \pi_i) + 2n_i (\hat \pi_i -\pi_i)(\pi_i-\bar \pi)$.
Using $\sum_{i=1}^k E(G_i^4) \leq  \sum_{i=1}^k
\left( n_i^4E((\hat \pi_i - \pi_i)^8) +  d_i^4  E( (\hat \pi_i (1-\hat \pi_i))^4) + 2^4 n_i^4 E(\hat \pi_i -\pi_i)^4
(\pi_i-\bar \pi)^4 \right)$.
From Lemma 4,  we have $n_i^4E((\hat \pi_i - \pi_i)^8)   \leq  O\left(\theta_i^4 + \frac{\theta_i}{n_i} \right)$,
$d_i^4  E( (\hat \pi_i (1-\hat \pi_i))^4) \leq 2^4 \left(\frac{3\theta_i^2}{n_i^2}  + \frac{(1-6\theta_i)\theta_i}{n_i^3} \right)
\leq O( \frac{\theta_i^2}{n_i^2} + \frac{\theta_i}{n_i^3} )$ where $O(\cdot)$ is uniform in $1\leq i \leq k$.
Using the result in Lemma 1, we have  $2^4 \sum_{i=1}^k n_i^4 E(\hat \pi_i -\pi_i)^4
\sum_{i=1}^k (\pi_i-\bar \pi)^4  \leq 2^4 \sum_{i=1}^k  n_i^4 (\pi_i-\bar \pi)^4 \left(\frac{3\theta_i^2}{n_i^2}+ \frac{(1-6\theta_i)\theta_i}{n_i^3} \right)
\leq   \max_{1\leq i \leq k} \left\{n_i (\pi_i-\bar \pi)^2 ( \theta_i + \frac{1}{n_i}) \right\}  \sum_{i=1}^k n_i (\pi_i-\bar \pi)^2 \theta_i
=   \max_{1\leq i \leq k} \left\{n_i (\pi_i-\bar \pi)^2 ( \theta_i + \frac{1}{n_i}) \right\}
||{\mathb \pi} - \bar {\mathb \pi} ||^2_{\theta{\bf n}}$.
Therefore, we have
\begin{eqnarray}
\frac{\sum_{i=1}^k E(G_i^4)}{Var(T_1)^2} &\leq& \frac{  \sum_{i=1}^k \left(\theta_i^4  + \frac{\theta_i}{n_i} \right) +      \max_{1\leq i \leq k} \left\{n_i (\pi_i-\bar \pi)^2 ( \theta_i + \frac{1}{n_i}) \right\}
||{\mathb \pi} - \bar {\mathb \pi} ||^2_{\theta{\bf n}} }{ ({\cal \nu}_1 + ||{\mathb \pi} - \bar {\mathb \pi} ||^2_{\theta{\bf n}})^2 }\\
&=&\frac{ \sum_{i=1}^k (\theta_i^3 + \frac{\theta_i}{n_i})}
{\left(\sum_{i=1}^k (\theta_i^2 + \frac{\theta_i}{n_i})\right)^2}
+ \frac{ \max_{1\leq i \leq k} \left\{n_i (\pi_i-\bar \pi)^2 ( \theta_i + \frac{1}{n_i}) \right\} }
{{\cal \nu} +  ||{\mathb \pi} - \bar {\mathb \pi} ||^2_{\theta{\bf n}} }\rightarrow 0
\end{eqnarray}
from the given conditions.

The negligibility of $T_2 = N(\hat {\bar \pi} - \bar \pi)^2$
can be proven by noting that
$\frac{NE(\hat {\bar \pi} - \bar \pi)^2}{\sqrt{Var(T_1)}} = \frac{ \bar \theta }{\sqrt{Var(T_1)}}
    =  \frac{1}{N} \frac{\sumk  n_i \theta_i}{\sqrt{Var(T_1)}}
    \asymp   \frac{\max_i \theta_i \sumk n_i} { N \sqrt{ {\cal V}_1 +  ||{\mathb \pi} - \bar {\mathb \pi} ||^2_{\theta {\bf n}}  } }$
    by  (\ref{eqn:asympVar}) from the condition $(i)$.
  This leads to
      $\left(  \frac{\max_i \theta_i^2 } {{ {\cal V}_1 +  ||{\mathb \pi} - \bar {\mathb \pi} ||^2_{\theta {\bf n}}  } }  \right)^{1/2}      \rightarrow 0$ from the condition $(ii)$, so we have
 $\frac{N(\hat {\bar \pi} - \bar \pi)^2}{\sqrt{Var(T_1)}} \rightarrow 0$  in probability.
Combining (I) and (II), we conclude $\frac{T - \sum_{i=1}^k n_i (\pi_i - \bar \pi)^2}{\sqrt{{Var(T_1)}}} \rightarrow N(0,1)$ in distribution.

\section{Proof of Theorem \ref{thm:powercomparison}}
\begin{enumerate}
\item Proof of 1 : We prove $\beta(T_{new2}) \geq \beta(T_{new1})$. For this, we only need to show that
${\cal V}_{1} \geq {\cal V}_{1*}$ from Corollary \ref{cor:cor1}.
Let  $f(x)= 2 x^2(1-x)^2 + \frac{x(1-x)}{n}$, then
 $f(x)$ is convex for $0 < x <  \frac{1}{2} -\frac{1}{\sqrt{3}} \sqrt{1+\frac{1}{n}}$
 since $f''(x) >0 $ for $0 < x <  \frac{1}{2} -\frac{1}{\sqrt{3}} \sqrt{1+\frac{1}{n}}$.
Furthermore,  ${\cal V}_1 = \sumk f(\pi_i)$ and ${\cal V}_{1*} = k f(\bar \pi)$ for
$\bar \pi = \frac{1}{N} \sumk n_i \pi_i$.
From the convexity of $f$,
if $n_i=n$ for all $1\leq i\leq k$, we have
$\frac{1}{k} {\cal V}_1 = \frac{1}{k} \sumk f(\pi_i) \geq  f(\bar \pi) = \frac{1}{k} {\cal V}_{1*}$.
Therefore, ${\cal V}_1 \geq {\cal V}_{1*}$ which leads to
$ \lim_{k\rightarrow \infty}(\beta(T_{new2}) - \beta(T_{new1})) \geq 0$
for the given $0 < \pi_i <  \frac{1}{2} -\frac{1}{\sqrt{3}}$ for all $i$.

Under the given condition,  $\hat {\cal B}_{0k} = 2k (1+o_p(1))$ and
\begin{eqnarray*}
T_{new2} &=&   \frac{ \sum_{i=1}^k n_i (\hat \pi_i -\hat {\bar \pi})^2 -   \sum_{i=1}^k \hat {\pi}_{i}(1-\hat {\pi}_i)  }{\sqrt{2k \hat {\bar \pi} (1-\hat {\bar \pi}) }} (1+o_p(1))\\
T_{\chi} &=&  \frac{ \sum_{i=1}^k n_i (\hat \pi_i -\hat {\bar \pi})^2 -   k \hat {\bar \pi}(1-\hat {\bar \pi})  }{\sqrt{2k \hat {\bar \pi} (1-\hat {\bar \pi}) }} (1+o_p(1))
\end{eqnarray*}
which leads to
\begin{eqnarray*}
T_{new2} - T_{\chi} =  \frac{k \hat {\bar \pi}(1-\hat {\bar \pi}) -   \sum_{i=1}^k \hat {\pi}_{i}(1-\hat {\pi}_i)}{\sqrt{2k \hat {\bar \pi} (1-\hat {\bar \pi})}} (1+o_p(1)).
\end{eqnarray*}
Using $k\hat {\bar \pi}(1-\hat {\bar \pi}) \geq \sum_{i=1}^k \hat \pi_i(1-\hat \pi_i)$,
 $  \lim_{k \rightarrow \infty} P( T_{new2} - T_{\chi} \geq 0) \rightarrow 1$ which leads to
$\lim_{k \rightarrow \infty} (\beta(T_{new2}) - \beta(T_{\chi})) \geq 0$.
\item
Proof of 2:
Note that ${\cal A}_{1i} = 2(1+o(1))$ and ${\cal A}_{2i} = 4(1+o(1))$
where $o(1)$ is uniform in $i$.
Using $\bar \pi = (k^{-\gamma} + \delta k^{\alpha-1}) (1+O(k^{-1}))$ and $\tilde \theta = \bar \pi (1+ o(1))$,  we obtain
\begin{eqnarray*}
{\cal V}_1 &=&  \left( 2 \sumk \theta_i^2  +  4 \sumk \frac{\theta_i}{n_i} \right) (1+o(1))   =
  \left(2 (k-1) k^{-2\gamma} + 2 (k^{-\gamma}+\delta)^2+   \frac{(k-1)k^{-\gamma}}{n} + \frac{k^{-\gamma} + \delta}{ n k^{\alpha}} \right) \\
&=&  (2k^{1-2\gamma} + 2 \delta^2 +  \frac{4k^{1-\gamma}}{n}) (1+o(1))\\
{\cal V}_{1*} &=&
 2 k(k^{-\gamma} + \delta k^{\alpha-1})^2 (1+O(k^{-1})) +  4 \tilde \theta \sumk \frac{1}{n_i} \\
   &=&
2 k^{1-2\gamma} + 4\delta k^{\alpha-\gamma} +2 \delta^2 k^{2\alpha-1} +
 4(k^{-\gamma} + \delta k^{\alpha -1}) \left( \frac{k-1}{n} + \frac{1}{ nk^{\alpha}} \right) (1+o(1)) \\
 &=&  2 k^{1-2\gamma} + 4 \delta k^{\alpha-\gamma} +2 \delta^2 k^{2\alpha-1}
 +  4 \frac{k^{1-\gamma} + \delta k^{\alpha}}{n} (1+o(1))
\end{eqnarray*}
so
\begin{eqnarray}
\frac{{\cal V}_{1*} - {\cal V}_1}{{\cal V}_1} =
\frac{ (2\delta k^{\alpha -\gamma} + \delta^2 (k^{2\alpha-1}-1)) + 2\frac{\delta k^{\alpha}}{n} (1+o(1))}
{k^{1-2\gamma} + 2 \delta^2 + 2 \frac{k^{1-\gamma}}{n} (1+o(1))}.
\label{eqn:Vratio}
\end{eqnarray}

\begin{enumerate}
\item if $\alpha + \gamma <1$ and $\alpha \geq \frac{1}{2}$,
then $k^{\alpha-\gamma} = o(k^{2\alpha-1})$, therefore
$(\ref{eqn:Vratio})  =  \frac{\delta^2 k^{2\alpha-1}  I(\alpha \neq \frac{1}{2}) + 2\frac{k^{\alpha}}{n}}{ k^{1-2\gamma} + \delta^2 + 2\frac{k^{1-\gamma}}{n}}
\rightarrow 0$
where $I(\cdot)$ is an indicator function.

\item if $\alpha+\gamma<1$, $\alpha < \frac{1}{2}$
 and $\alpha \geq \gamma$, then
 $(\ref{eqn:Vratio}) = \frac{2 \delta k^{\alpha-\gamma}-\delta^2 + 2\frac{k^{\alpha}}{n}}{k^{1-2\gamma} + \delta^2 + 2\frac{k^{1-\gamma}}{n}} \rightarrow 0$.

 \item if $\alpha+\gamma <1$, $\alpha < \frac{1}{2}$, $\gamma \leq \frac{1}{2}$ and $\alpha<\gamma$,
 then $(\ref{eqn:Vratio})= \frac{-\delta^2 + 2\frac{k^{\alpha}}{n}}{ k^{1-2\gamma} + \delta^2 + 2\frac{k^{1-\gamma}}{n}} \rightarrow 0$.

 \item if $\alpha+\gamma <1$, $\alpha < \frac{1}{2}$ and $\gamma > \frac{1}{2}$, then there are two cases depending on
 the behavior of $n$.
 When $\frac{k^{1-\gamma}}{ n} \rightarrow 0$, then
 $(\ref{eqn:Vratio}) \rightarrow \frac{-\delta^2}{\delta^2} = -1$.
 When $\frac{k^{1-\gamma}}{n} \rightarrow \infty$,
 $(\ref{eqn:Vratio}) = \frac{\frac{k^{\alpha}}{n}}{\frac{k^{1-\gamma}}{n}} (1+o(1)) = k^{\alpha+\gamma -1} \rightarrow 0$.

 \item if $\alpha+\gamma >1$, $\alpha>\frac{1}{2}$ and $\gamma < \frac{1}{2}$, then
 $(\ref{eqn:Vratio}) = \frac{\delta^2 k^{2\alpha-1} + 2\frac{k^{\alpha}}{n}}{ k^{1-2\gamma} + 2\frac{k^{1-\gamma}}{n}} (1+o(1))
 \rightarrow \infty$.

 \item if $\alpha+\gamma >1$, $\alpha>\frac{1}{2}$ and $\gamma \geq \frac{1}{2}$, then
  $(\ref{eqn:Vratio}) =   \frac{ k^{\alpha-\gamma}  + \delta^2 k^{2\alpha-1} + 2\frac{k^{\alpha}}{n} }{ I(\gamma=\frac{1}{2}) +\delta^2 + 2\frac{k^{1-\gamma}}{n}}  (1+o(1)) \rightarrow \infty$.

 \item  if $\alpha+\gamma >1$, $\alpha<\frac{1}{2}$ and $\gamma>\frac{1}{2}$,
 then   $\alpha<\gamma$ and  $(\ref{eqn:Vratio}) =   \frac{ - \delta^2  + 2\frac{k^{\alpha}}{n} }{  \delta^2 + 2\frac{k^{1-\gamma}}{n}} (1+o(1)) = \frac{-\delta^2 + \frac{k^{\alpha}}{n}}{ \delta^2 + 2\frac{k^{1-\gamma}}{n}} (1+o(1))$.
 There are two situations depending $n$.
When $\frac{k^{\alpha}}{n} \rightarrow \infty$,  $(\ref{eqn:Vratio}) =
\frac{-\delta^2 + \frac{k^{\alpha}}{n}}{ 2\delta^2 + \frac{k^{1-\gamma}}{n}} (1+o(1)) \rightarrow \infty$.
When    $\frac{k^{\alpha}}{n} \rightarrow 0$,  we have  $\frac{k^{1-\gamma}}{n} \rightarrow 0$, so we derive
$(\ref{eqn:Vratio}) =
\frac{-\delta^2}{ \delta^2 } (1+o(1)) \rightarrow -1$.
\end{enumerate}
In $ (a) \cup (b) \cup (c) = \{ (\alpha,\gamma) :  0<\alpha<1, 0<\gamma<1,
 0< \alpha + \gamma <1, 0< \gamma \leq \frac{1}{2} \}$,  we have
$\lim_n \frac{{\cal V}_{1*}}{{\cal V}_1} =1$ leading to
$\lim_n (\beta(T_{new1}) - \beta(T_{new1}))=0$.
In  $(e)\cup (f) = \{(\alpha, \gamma) : 0<\alpha<1, 0<\gamma<1, \alpha + \gamma > 1, 1>\alpha>\frac{1}{2} \}$,
 we have $ \lim \frac{{\cal V}_{1*}}{{\cal V}} >1$ which leads to
 $\lim_n (\beta(T_{new1}) - \beta(T_{new2})) >0$.

In (e) and (g), the performances are different depending on the sample sizes.

\item We first have
\begin{eqnarray*}
{\cal V}_1 &=&  2 (k^{-\gamma} + \delta)^2  + 2 (k -1) k^{-2\gamma}  +  \frac{4(\delta + k^{-\gamma})} {n}
+ 4 (k-1)\frac{k^{-\gamma}}{ nk^{\alpha}} \\
&=& \left(2 \delta^2 + 2 k^{1-2\gamma} + \frac{4 k^{1-\gamma -\alpha}}{n} \right)(1+o(1)).
\end{eqnarray*}
Since $\tilde \theta = \bar \pi (1-\bar \pi) = \frac{\delta + k^{\alpha -\gamma +1}}{k^{\alpha+1}} (1+o(1)) = k^{-\gamma}(1+o(1))$
from $  0< \alpha <1$ and $0<\gamma <1$,
\begin{eqnarray*}
{\cal V}_1^* &=&  2 k^{1-2\gamma}  +   \frac{4k^{-\gamma}}{n} + \frac{ (k-1)k^{-\gamma} }{nk^{\alpha}}      \\
&=& \left(2 k^{1-2\gamma} + \frac{4 k^{1-\gamma -\alpha}}{n} \right)(1+o(1)).
\end{eqnarray*}
If $1-2\gamma <0$ and $ k^{1-\gamma -\alpha} =o(n)$, then
${\cal V}_1 = \delta^2 (1+o(1))$ and ${\cal V}_1^* =o(1)$,
we have $\frac{{\cal V}_1}{ {\cal V}_1^* } \rightarrow \infty$ which leads to  $ \beta(T_{new2}) - \beta(T_{new1}) >0$.
\end{enumerate}
\end{document}